\documentclass[12pt]{amsart}
\usepackage{amsmath,amssymb,amscd,amsthm,amsfonts}
\usepackage[all,cmtip]{xy} 
\usepackage{hyperref}
\usepackage [numbers] {natbib}
\usepackage{enumerate}
\usepackage{mathtools}

\newtheorem{thm}{Theorem}[section]

\newtheorem{lemma}[thm]{Lemma}
\newtheorem{prop}[thm]{Proposition}
\newtheorem{cor}[thm]{Corollary} 
\theoremstyle{definition}
 
\theoremstyle{remark}

\newtheorem{rmk}[thm]{Remark}
\numberwithin{equation}{section}

\newcommand{\fa}{\mathfrak{a}}

\renewcommand{\AA}{\mathbb{A}}

\newcommand{\RR}{\mathbb{R}} 
\newcommand{\ZZ}{\mathbb{Z}}

\newcommand{\A}{\mathbb{A}}

\newcommand{\cA}{\mathcal{A}}

\newcommand{\cE}{\mathcal{E}}

\newcommand{\cP}{\mathcal{P}}

\newcommand{\cS}{\mathcal{S}}

\newcommand{\cusp}{\mathrm{cusp}}

\newcommand{\half}{\frac {1} {2}}

\DeclareMathOperator{\GL}{\mathrm{GL}}
\DeclareMathOperator{\RGL}{\mathrm{RGL}}

\DeclareMathOperator{\Herm}{\mathrm{Herm}}

\DeclareMathOperator{\Mp}{\mathrm{Mp}}

\DeclareMathOperator{\tr}{\mathrm{tr}}

\DeclareMathOperator{\Hom}{\mathrm{Hom}}

\DeclareMathOperator{\Res}{\mathrm{Res}} 
\DeclareMathOperator{\res}{\mathrm{res}} 
\newcommand{\isom}{\cong}

\DeclareMathOperator{\Nm}{\mathrm{Nm}}

\DeclareMathOperator{\vol}{\mathrm{vol}}

\DeclareMathOperator{\lmod}{\backslash}

\DeclareMathOperator{\LO}{\mathrm{LO}}
\DeclareMathOperator{\FO}{\mathrm{FO}} 
\let\Re\undefined
\DeclareMathOperator{\Re}{\mathrm{Re}}

\newcommand{\form}[2]{\langle{#1},{#2}\rangle}

\renewcommand{\RGL}{R_{E/F}\GL}

\begin{document}

\title[Poles of Eisenstein series and theta lifts] {Poles of Eisenstein series and theta lifts for unitary groups}
\author{Chenyan Wu}
\email{chenyan.wu@unimelb.edu.au}
\address{School of Mathematics and Statistics
The University of Melbourne, Victoria 3010, Australia}

\thanks{Corresponding author.}
\date{}

\begin{abstract}
  We derive a precise relation of poles of  Eisenstein series associated to the cuspidal datum $\chi\otimes\sigma$ and the lowest occurrence of theta lifts of a cuspidal automorphic representation $\sigma$ of a unitary group, where $\chi$ is a conjugate self-dual character. A key ingredient of the proof is the computation of period integrals of truncated Eisenstein series. 
  \end{abstract}
\maketitle{}

\section*{Introduction}
\label{sec:introduction}

There has been extensive study on the invariants, such as the $L$-functions and the first occurrence indices of the theta lifts, attached to cuspidal representations of classical groups or of their covering groups. For concreteness, we will give an overview only for the unitary case, as this is the focus of this paper, though we will still list  references for other cases in this introduction. 

Let $F$ be a number field and $E$ a quadratic field extension of $F$. Let $\AA:=\AA_F$ and $\AA_E$ be the corresponding adele rings.   Let $X$ be a skew-Hermitian space over $E$ and let $G(X)$ be the unitary group associated to $X$.  Let $Y$ be a Hermitian space over $E$ and let $G(Y)$ be the unitary group associated to $Y$. Then $G(X)$ and $G(Y)$ form a dual reductive pair and we can define an automorphic representation of $G(Y)(\AA)$ called the (global) theta lift of a cuspidal automorphic representation of $G(X)(\AA)$. We can also form the theta lift in the opposite direction.
 Let $\sigma$ be a cuspidal representation of  $G(X)(\AA)$.  The Rallis inner product formula\cite{MR1289491,MR3279536} shows that
the poles of the $L$-function  $L(s,\sigma\times\chi^{-1})$  are related to the non-vanishing of theta lifts of $\sigma$. Here $\chi$ is a conjugate self-dual automorphic character of $\AA_E^\times$, which is needed to define the theta lifts. It is  derived from the doubling  integral construction of $L$-functions\cite{ps-rallis87_l_funct_for_the_class_group} (see also \cite{MR3006697,MR3211043}), in which the Siegel Eisenstein series play a key role and the regularised Siegel-Weil formula\cite{MR1289491,MR1863861,MR2064052}, which expresses the regularised theta integrals as values or residues of the Siegel Eisenstein series. The boundary case of the Siegel-Weil formula was treated in \cite{MR3595433} and the Siegel-Weil formula was extended to the `second term range' by \cite{MR3279536}. 

Due to the presence of ramified local factors in the Rallis inner product formula, the relation between the locations of the poles of $L(s,\sigma\times\chi^{-1})$ and the non-vanishing of theta lifts is not exact. One way to control the local factors in as in \cite{MR3211043}. The poles of the completed $L$-functions were interpreted as the obstruction to the local-global principle of theta lift. 

We will follow an idea of M\oe glin's\cite{MR1473165} to  consider the poles of  the  Eisenstein series associated to the cuspidal datum $\chi\otimes\sigma$. These Eisenstein series are closely related to the $L$-functions $L(s,\sigma\times\chi^{-1})$. We note that these are not Siegel Eisenstein series. It can be  shown that the location of the maximal pole of the Eisenstein series we consider  characterises completely the lowest occurrence index of theta lifts of $\sigma$. The cases of the orthogonal and symplectic/metaplectic dual reductive pairs were treated in \cite{MR1473165,MR2540878,wu:_period-metaplectic}. We started the investigation of the case of  unitary dual reductive pairs in \cite{MR3435720} which worked out many important preliminary results and the aim of this paper is to  complete the results. We also correct some misstatement and typos in \cite{MR3435720}.

To state our main theorem, first we explain briefly the definitions of first occurrence index and lowest occurrence index of theta lifts. Assume that $Y$ is an anisotropic Hermitian space, so that it sits at the bottom of its Witt tower.
For $a\in \ZZ_{\ge 0}$, we form Hermitian spaces $Y_a$ by adjoining $a$ copies of the hyperbolic plane to $Y$. The cuspidal representation $\sigma$ of $G(X)(\AA)$ can be lifted to $G(Y_a)(\AA)$ for each $a$. The theta lifts do not vanish for $a$ large enough\cite[Theorem~I.2.1]{MR743016}. The first occurrence index  $\FO_{\psi,\chi}^Y(\sigma)$ for the Witt tower of $Y$ then is defined to be the minimal dimension of $Y_a$ for which the theta lift to $G(Y_a)$ does not vanish. For the lowest occurrence index $\LO_{\psi,\chi} (\sigma)$, we further minimise across all Witt towers that are compatible with $\chi$. The subscript $\psi$ denotes the additive character that is used in defining theta lifts.
See Sec.~\ref{sec:theta-corr-unit} for more details.

The other main ingredient is the Eisenstein series associated to the cuspidal datum $\chi\otimes\sigma$. Form the skew-Hermitian space $X_1$ by adjoining a copy of the hyperbolic plane  to $X$. We polarise  the hyperbolic plane as $\ell_1^+\oplus\ell_1^-$. Let $Q_1$ be the parabolic subgroup of $G(X_1)$ that stabilises the isotropic line $\ell_1^-$. Let $M_1$ be its standard Levi subgroup and $N_1$ its unipotent radical.  Let $\cA_1(s,\chi,\sigma)$ denote the space of automorphic forms on $N_1(\AA)M_1(F)\lmod G(X_1)(\AA)$ that transform according to $\chi|\ |^s\otimes\sigma$ on $M_1(F)\lmod M_1(\AA)$.  See Sec.~\ref{sec:eisenstein-series} for more details. Then we form the Eisenstein $E(g,f_s)$ for each section $f_s$ in $\cA_1(s,\chi,\sigma)$.

We can now state our main theorem:
\begin{thm}[{Cor.~\ref{cor:Eis-pole=LO}}]
Let  $\sigma$ be a cuspidal representation of $G(X)(\AA)$. Then  $s_0$ is the maximal pole of the Eisenstein series $E(g,f_s)$ for  $f_s$ running over $\cA_1(s,\chi,\sigma)$ if and only if $\LO_{\psi,\chi} (\sigma) =  \dim X + 1 - 2 s_0$.
\end{thm}
This theorem strengthens Theorem~3.1 in \cite{MR3435720} where we  achieved the inequality that $\LO_{\psi,\chi} (\sigma)\le \dim X + 1 - 2 s_0$.  In this paper, we derive the other inequality  from an intricate result (Cor.~\ref{cor:FO->period-residue-Eis})  on the relation between period integrals of residues of Eisenstein series twisted by a theta series and the first occurrence index $\FO_{\psi,\chi}^Y(\sigma)$.  We note that the first occurrence is a finer invariant than the lowest occurrence just as non-vanishing of period integrals of residues is a finer invariant than having a pole.

More precisely, the period integrals involved are of the form
\begin{align*}
  \int_{[G(Z_1)]} \cE_{-1}(g,f_{s_1}) \overline {\theta_{\psi,\chi,\chi_2,X_1,Y} (g,1,\Phi)} dg
\end{align*}
where $\cE_{-1}(g,f_{s_1})$ denotes the residue of $E(g,f_s)$ at $s=s_1$ and $Z_1$ is a non-degenerate skew-Hermitian subspace of $X_1$. We refer the reader to Sec.~\ref{sec:theta-corr-unit} for unexplained notation.
We can show that if we know the first occurrence index for the Witt tower of $Y$, then the period integral is non-vanishing for appropriate $s_1$ and $Z_1$. This implies the other inequality.

As the bulk of the article is for computing the period integral, we describe our approach here.
We start with a period integral of a truncated Eisenstein series and  cut it into many pieces according to orbits in a certain flag variety. Each piece turns out to have a period integral of the form
\begin{align*}
  \int_{[J]} f(g)\overline{\theta_{\psi,\chi,\chi_2,X,Y}(g,1,\Phi)}dg,
\end{align*}
as an inner integral,
where $f\in\sigma$ and $J$ is a  subgroup of $G(X)$. 
We can show that most of the pieces must vanish given our condition on $\FO_{\psi,\chi}^Y(\sigma)$. Finally we show that the non-vanishing of certain Fourier coefficients of the  theta lifts of $\sigma$ implies that the remaining pieces do not vanish.
The computation is justified by showing absolute convergence through  delicate analysis. 

We remark  that each case of the dual reductive pairs has its own peculiarity and presents its own challenge. In the symplectic/metaplectic case, our method led to  period integrals on Jacobi groups which account for the lack of `odd symplectic groups', while in the orthogonal and unitary cases we get  period integrals on certain homogeneous spaces  to account for the fact that orthogonal groups and unitary groups may not be quasi-split. We are able to prove some very technical results on absolute convergence of such  integrals, which may be useful in future studies involving these types of period integrals. Our arguments are greatly inspired by those in  \cite{MR2540878} which treats the orthogonal case. However as \cite{MR2540878} did not provide sufficient arguments for showing 
the absolute convergence of the period integrals on certain homogeneous spaces mentioned above, we worked out some new arguments and they  can be adapted to the orthogonal case easily. (See Sec.~\ref{sec:absol-conv-i_omega22}.)

According to the $(\tau,b)$-theory proposed in \cite{jiang14:_autom_integ_trans_class_group_i}, theta series are the kernel functions  which give the endoscopic lift between the Arthur packet attached to the parameter $\Psi_2$ and that attached to the parameter $\Psi$ where $\Psi=\tau\boxtimes \nu_b \boxplus \Psi_2$ in the special case of $\tau=\chi$ of \cite[Principle 1.2]{jiang14:_autom_integ_trans_class_group_i}. In general, $\tau$ can be a conjugate self-dual unitary cuspidal automorphic representation of $\GL_n(\AA_E)$. We refer the reader to \cite{MR3135650,MR3338302,kaletha:_endos_class_repres} for  Arthur's theory of endoscopy. Our future work will investigate more general kernel functions and the properties of the  lifts they afford.

We give an overview of the structure of this article. In Sec.~\ref{sec:notation-preliminary} we set up some notation for the Eisenstein series and theta lifts. In Sec.~\ref{sec:summary-results-from-part-1}, we recall a result, which we will strengthen, from \cite{MR3435720}. We also recall a result 
from \cite{MR3071813} on theta lift. As the definition of theta lift there did not use complex conjugation, we  rephrase  the results to conform to our present definition.  Sec.~\ref{sec:peri-integr-main} contains the statements of the main results: relation of period integrals and first occurrence indices. We show how this gives us a way to strengthen \cite[Theorem 3.1]{MR3435720}. We define period integrals of Eisenstein series (twisted by certain theta functions) on subgroups of $G(X_1)$ and apply the Arthur truncation method to regularise the period integrals. We cut the period integrals into many pieces according to orbits in a certain flag variety. In Sec.~\ref{sec:comp-integrals}, we evaluate each piece and in Sec.~\ref{sec:absol-conv-integr}, we justify all the changes of order of integration and summation by showing absolute convergence of the pieces.

 \section*{Acknowledgement}
\label{sec:acknowledgement}
The author would like to thank her post-doc mentor, Professor Dihua Jiang, for introducing her to the beautiful world of automorphic representations and for the many  suggestions of which she generally only realises the relevance and usefulness later. The author would like to thank her former colleague, Ronggang Shi, at SCMS, for encouraging her to keep tackling the issue of absolute convergence. The author is grateful to the University of Melbourne for promptly supplying her with a nice new laptop when her old laptop stopped working during lock-down. The author would also like to thank the referee for the careful reading and the many helpful suggestions that make the article more accurate.

\section{Notation and Preliminary}
\label{sec:notation-preliminary}

We will use similar notation as in our previous paper\cite{MR3435720}. Let $F$ be a number field and $E$ a quadratic field extension of $F$.  Fix a traceless element $\daleth$ in $E^\times$. Thus if $c$ denotes the non-trivial element in the Galois group of $E/F$, then $\daleth^c = -\daleth$. Let $\AA=\AA_F$ be the ring of adeles of $F$ and $\AA_E$ that of $E$. Let $\varepsilon_{E/F}$ be the quadratic character of $\AA_F^\times$ that is associated to the quadratic extension $E/F$ via the class field theory.

Let $X$ be a skew-Hermitian space  and $Y$ a Hermitian space over $E$. Then we have the unitary groups $G
(X)$ and $G (Y)$  defined over $F$ and they form a dual reductive pair. Fix a pair of automorphic characters
$(\chi_1,\chi_2)$ of $\AA_E^\times$ such that $\chi_1|_{\AA_F^\times} =
\varepsilon_{E/F}^{\dim Y}$ and $\chi_2|_{\AA_F^\times} =
\varepsilon_{E/F}^{\dim X}$. As the parity of the exponent plays a significant role in the results below, we set $\epsilon_{\chi}=a$ for an automorphic character $\chi$ of $\AA_E^\times$, if $\chi|_{\AA_F^\times} = \varepsilon_{E/F}^a$ for $a=0$ or $1$. We also fix a non-trivial additive character $\psi$ of $F\lmod \AA$ for use in the Weil representation. We get an additive character of $E\lmod \AA_E$ given by $\psi_E := \psi\circ \half\tr_{E/F}$.

For a non-negative integer $a$, set $X_a=\ell_a^+\oplus X \oplus \ell_a^-$ to be the skew-Hermitian space formed by adjoining a $2a$-dimensional polarised split skew-Hermitian space $\ell_a^+\oplus \ell_a^-$ to $X$. If $X$ can be decomposed as $\ell_a^+\oplus X' \oplus \ell_a^-$ for some skew-Hermitian subspace $X'$ of $X$, then we write $X_{-a}$ for $X'$. We have analogous constructions for the Hermitian space $Y$. In this paper, $a$ is usually $1$ and we let $e_1^+$ (resp. $e_1^-$) be a basis element for $\ell_1^+$ (resp. $\ell_1^-$)  such that $\form{e_1^+}{e_1^-} =1$. The angled brackets denote the pairings on the (skew-)Hermitian spaces. Sometimes subscripts are added to indicate which space we are concerned with.

Let $Q_a$ be the parabolic subgroup of $G(X_a)$ that stabilises $\ell_a^-$. Let $M_a$ be its standard Levi subgroup and $N_a$  its unipotent radical. We have
\begin{equation*}
  M_a \isom R_{E/F}\GL_a \times G(X)
\end{equation*}
where $R_{E/F}$ denotes the restriction of scalars of Weil.  For $t\in\RGL_a$, set $m_a(t)$ 
to be the corresponding element in $M_a$. Let $\rho_{Q_a}$ denote the half sum of positive roots in $N_a$. It is an element in $\fa_{M_a}^*$, which is a one-dimensional real vector space in the notation of \cite{MR518111} (see also \cite{MR1361168}). Under the Shahidi normalisation\cite{MR2683009}, we may view it as the (half) integer $(\dim X + a)/2$. For more details, see \cite[Sec.~2.2]{MR3435720}.

Fix a good maximal compact subgroup $K_{G(X_a)}$ for $G(X_a)(\AA)$ such that the Iwasawa decomposition
\begin{equation*}
  G(X_a)(\AA) = Q_a(\AA) K_{G(X_a)}
\end{equation*}
holds.
For $g\in G(X_a)(\AA)$, we  decompose it according to the Iwasawa decomposition $g=n_g m_a(t_g) h_g k_g$ for some $n_g\in N_a(\AA)$, $t_g\in\RGL_a(\AA)$, $h_g\in G(X)(\AA)$ and $k_g\in K_{G(X_a)}$. 

\subsection{Eisenstein Series}
\label{sec:eisenstein-series}

Write $\cA_\cusp(G(X))$ for the set of irreducible cuspidal automorphic representations of $G(X)(\AA)$. Let $\chi$ be a character of $E^\times\lmod \AA_E^\times$. Most often $\chi$ will be equal to $\chi_1$. Let $\cA_a(s,\chi,\sigma)$ be the space of smooth right $K_{G(X_a)}$-finite  functions $f_s$ on $N_a(\AA)M_a(F)\lmod G(X_a)(\AA)$ such that
\begin{enumerate}
\item for any $x\in \RGL_a (\A)$ and $g\ \in G (X_a) (\A)$,
  \begin{equation*}
    f_s (m_a (x) g) = \chi (\det x) |\det x|_{\A_E}^{s+\rho_{Q_a}} f_s (g);
  \end{equation*}
\item for any fixed $k\in K_{G(X_a)}$ and for any  $h \in G (X) (\A)$, the function
  \begin{equation*}
    h \mapsto f_s ( h k)
  \end{equation*}
is a smooth right $K_{G(X_a)}\cap M_a(\AA)^1$-finite vector in the space of $\sigma$.
\end{enumerate}
The superscript $1$ means taking the intersection of the kernels of characters as in \cite[I.1.4]{MR1361168}.

Given $f_{s_0}\in \cA_a(s_0,\chi,\sigma)$, we can extend it to a section $f_s\in \cA_a(s,\chi,\sigma)$ by setting
\begin{equation*}
  f_s(g) = |\det (t_g)|^{s-s_0} f_{s_0}(g)
\end{equation*}
where $t_g\in\RGL_a(\AA)$ appears in the Iwasawa decomposition of $g$. We note that $t_g$ is not uniquely determined in the Iwasawa decomposition, but the expression $|\det (t_g)|$ is independent of the choice.
We form the Eisenstein series on $G(X_a)(\AA)$ associated to $f_s$ as follows:
\begin{equation*}
  E(g,f_s) := \sum_{\gamma\in Q_a(F)\lmod G(X_a)(F)} f_s(\gamma g).
\end{equation*}
By Langlands' theory of Eisenstein series, it is absolutely convergent
for $\Re (s)> \rho_{Q_a}$ and has meromorphic continuation to
the whole $s$-plane, with finitely many poles in the half plane $\Re
(s) >0$, which are all real in our case (\cite[IV.1]{MR1361168}).  Let $\cP_a(\chi,\sigma)$ be the set of positive poles of these Eisenstein series. Some results on poles of these Eisenstein series were derived in \cite{MR3435720}.

\subsection{Theta correspondence, first occurrence and lowest occurrence}
\label{sec:theta-corr-unit}

The main difference to the definitions we used in \cite{MR3435720} is that we use the complex conjugate version of theta lift. In this way, we can unify the results with the metaplectic case better. This will also make our result\cite{MR3071813} on the involutive property of theta lifts look more natural, as we will not need to switch between using $\psi$ and $\psi^{-1}$ for the additive character that is used in the Weil representations.

 Along with the additive character $\psi$,
$\chi_1$ (resp. $\chi_2$) determines a splitting of the metaplectic group $\Mp(R_{E/F} ( Y\otimes_EX) ) (\A)$ 
over $G (X) (\A)$ (resp. $G (Y)(\A)$). Then we can define the Weil representation
$\omega_{\psi,\chi_1,\chi_2}$ of $G (X) (\A) \times G (Y) (\A)$ on the 
Schwartz space $\cS (V (\A))$ where
$V$ is a maximal isotropic subspace of
$R_{E/F}(Y\otimes_EX)$.  When we do not want to emphasise which maximal isotropic subspace we use, we will write $\cS_{X,Y}(\AA)$ for $\cS (V (\A))$. See \cite{MR1286835} for more details.

If we change $\chi_2$ to another eligible character $\chi_2'$,
then (c.f. \cite [eq.~(1.8)] {MR1327161} and \cite [eq.~(2.9)] {MR3071813})
\begin{equation}
  \label{eq:change-chi2}
  \omega_{\psi,\chi_1,\chi_2'} (g,h) = \nu_{\chi_2'\chi_2^{-1}} (\det h)\omega_{\psi,\chi_1,\chi_2} (g,h)
\end{equation}
for $(g,h) \in  G (X) (\A) \times G (Y) (\A)$. The character $\nu_{\chi_2'\chi_2^{-1}}$ is defined in \eqref{eq:nu_chi}. A similar result holds if we change $\chi_1$. 
Assume that $\epsilon_\chi=0$. We may associate to $\chi$ a character $\nu_\chi$ which at each local place $v$ of $E$ is a character of $E_v^1$ defined by
\begin{equation}\label{eq:nu_chi}
  \nu_{\chi, v} (x/\bar {x}) = \chi_v (x).
\end{equation}
In other words,  $\chi$ is the base change of $\nu_{\chi}$.
  
For $\phi\in \cS (V (\A))$ we form the theta series
\begin{equation*}
  \theta_{\psi,\chi_1,\chi_2,X,Y} (g,h,\phi) :=
\sum_{v\in V (F)} \omega_{\psi,\chi_1,\chi_2} (g,h)\phi (v)
\end{equation*}
for $g\in G (X) (\A)$ and $h\in G (Y) (\A)$.
For $\sigma\in\cA_\cusp(G(Y))$,  $f\in\sigma$ and $\phi\in \cS (V (\A))$, we define
\begin{equation*}
  \theta_{\psi,\chi_1,\chi_2,Y}^X (g,f,\phi) := \int_{[G (Y)]} \theta_{\psi,\chi_1,\chi_2,X,Y} (g,h,\phi) \overline {f (h)} dh.
\end{equation*}
Then the  theta lift $\theta_{\psi,\chi_1,\chi_2, Y}^{X} (\sigma)$  is defined to be the span of all such 
$\theta_{\psi,\chi_1,\chi_2,Y}^X (g,f,\phi)$'s. 
Similarly if $\sigma\in\cA_\cusp(G(X))$, $\theta_{\psi,\chi_1,\chi_2, X}^{Y} (\sigma)$ is defined to be the span of
\begin{equation*}
  \theta_{\psi,\chi_1,\chi_2,X}^Y (h,f,\phi) := \int_{[G (X)]} \theta_{\psi,\chi_1,\chi_2,X,Y} (g,h,\phi) \overline {f (g)} dg
\end{equation*}
for all  $f\in\sigma$ and $\phi\in \cS (V (\A))$. We note that the choice of $V$ is not important, as the Weil representations on $\cS(V(\AA))$ for different choices of $V$ are related via the Fourier transform.

Now we return to the case $\sigma\in\cA_\cusp(G(X))$. Fix an anisotropic Hermitian space $Y$ (so that it lies at the bottom of its Witt tower). Define the first occurrence index of $\sigma$ with respect to $\chi$ in the Witt tower associated to $Y$:
\begin{equation*} 
  \FO_{\psi,\chi}^Y(\sigma) := \min \{\dim Y_a | \theta_{\psi,\chi,\chi_2, X}^{Y_a} (\sigma) \neq 0 \}.
\end{equation*}
Of course, $\chi$ should be compatible with $Y$, i.e., $\epsilon_\chi \equiv \dim Y \pmod 2$. The first occurrence index does not depend on the choice of $\chi_2$.
Define the lowest occurrence index of $\sigma$ with respect to $\chi$:
\begin{equation*}
  \LO_{\psi,\chi}(\sigma) := \min \{\FO_{\psi,\chi}^Y(\sigma) |   \epsilon_\chi \equiv \dim Y \pmod 2\}.
\end{equation*}
This means we further minimise across all the Witt towers that are compatible with $\chi$.

\section{Some relevant results on theta lifts}
\label{sec:summary-results-from-part-1}

We recall a result in \cite{MR3435720} which we aim to strengthen in this paper. We  take this chance to
correct an error in the statements which have some extraneous character twists. We also rephrase a result from \cite{MR3071813}, which figured in the proof of \cite[Theorem~3.1]{MR3435720},  using our current definition of theta lift, namely the conjugate version, for future reference. Note that, due to the presence of complex conjugation, the theta lifts  use the same additive character in both directions in our rephrased version.

\begin{thm}[{\cite[Theorem~5.1]{MR3071813}}] \label{thm:theta_back_and_forth}
  Let $\sigma\in\cA_\cusp(G(X))$ and let $Y$ be a Hermitian space. Assume that $\theta_{\psi,\chi_1,\chi_2,X}^Y(\sigma)$ is non-vanishing and
  cuspidal. Then
  \begin{enumerate}
  \item $\theta_{\psi,\chi_1,\chi_2,Y}^X(\chi_2\cdot\theta_{\psi,\chi_1,\chi_2,X}^Y(\sigma))=\chi_1\sigma$;
  \item $\theta_{\psi,\chi_1,\chi_2,Y}^{X_a}(\chi_2\cdot\theta_{\psi,\chi_1,\chi_2,X}^Y(\sigma))$ is
    orthogonal to all cusp forms on $G(X_a) (\AA$) for $a>0$;
  \item $\theta_{\psi,\chi_1,\chi_2,Y}^{X_{-b}}(\chi_2\cdot\theta_{\psi,\chi_1,\chi_2,X}^Y(\sigma))=0$ for $b>0$.
  \end{enumerate}
\end{thm}
To conserve space, we wrote $\chi$ for $\chi\circ\det$.

\begin{thm}[{\cite[Thm.~3.1, Prop.~2.8]{MR3435720}}]\label{thm:Eis-pole-LO}
For $\sigma\in\cA_\cusp(G(X))$, if $s_0$ is the maximal element of $\cP_1(\chi,\sigma)$, then
  \begin{enumerate}
  \item $s_0 = \frac{1}{2} (\dim X+1-\epsilon_\chi)-j$ for
    some integer $j$ such that $0 \le j < \frac{1}{2} (\dim X+1-\epsilon_\chi)$ and $s_0=
    \frac{1}{2} (\dim X+1)$ can be achieved only when  $\epsilon_\chi=0$, $X$ is anisotropic and $\sigma=\nu_\chi\circ \det$;
  \item \label{item:LO-pole-ineq}$\LO_{\psi,\chi} (\sigma) \le  2j+\epsilon_\chi = \dim X + 1 - 2 s_0$;
    \item $2j+\epsilon_\chi = \dim X + 1 - 2 s_0 \ge r_X$ where $r_X$ is the Witt index of $X$.
  \end{enumerate}
\end{thm}

We say a few words regarding the proof as we missed some characters in some of the expressions in \cite{MR3435720}. We took $f\in\chi^{-1}\sigma$ and formed the section $\Phi_{f,s}$ which lies in $\cA_a(s,\chi,\sigma)$. Using our current definition of theta lift, the first unnumbered equation below \cite[(4.3)]{MR3435720} should read as
\begin{equation*}
  \int_{[G(Y)]} \chi_2^{-1}(\det h)  \theta_{\psi,\chi,\chi_2,X_a,Y}(g_a,h,\phi_Y^{(1)}) \chi(\det g) \overline{\theta_{\psi,\chi,\chi_2,X,Y}(g,h,\overline{\phi_Y^{(2)}})} dh.
\end{equation*}
First we have added more subscripts for clarification. Second we have put in the missing factor $\chi(\det g)$. This factor should be there by \cite[Prop.~2.2]{MR1327161}. With this, we find that the residue of
the Eisenstein series $E(g_a,\Phi_{f,s})$ at $s=s_0+\half (a-1)$ is given by
\begin{equation*}
  \int_{[G(Y)]} \chi_2^{-1}(\det h)\theta_{\psi,\chi,\chi_2,X_a,Y}(g_a,h,\phi_Y^{(1)})  \int_{[G(X)]}  \overline{\theta_{\psi,\chi,\chi_2,X,Y}(g,h,\overline{\phi_Y^{(2)}})} \chi(\det g) f(g) dg dh.
\end{equation*}
Thus non-vanishing of the residue means that the inner integral
is non-vanishing, but that is exactly (the complex conjugate of) the theta lift of $\chi(\det g) f(g) \in \sigma$ from $G(X)(\AA)$ to $G(Y)(\AA)$.

The following theorem is a corollary to Thm.~\ref{thm:Eis-pole-LO} and we remove the extraneous $+1$ in the statements. Note that in \cite{MR3435720}, we followed the notation of \cite{MR2767519} for the $L$-functions and here we use the more common one such that $L^S (s,\sigma \times \chi) = L^S (s,BC(\sigma) \otimes \chi) $ where $BC$ denotes the standard base change.
\begin{thm}[{\cite[Thm.~3.3]{MR3435720}}]\label{thm:poleL}
For $\sigma\in\cA_\cusp(G(X))$, the following hold.
  \begin{enumerate}
  \item Assume either that the partial $L$-function $L^S (s,\sigma \times \chi^{-1})$  has a pole at $s= \half (\dim X+1-\epsilon_\chi) - k>0$,  for some integer $k$, or assume that $\dim X-\epsilon_\chi$ is even and that  $L^S (s,\sigma \times \chi^{-1})$ is non-vanishing at $s=1/2$, in which case we set $k=\half (\dim X-\epsilon_\chi)$. Then
    $\LO_{\psi,\chi} (\sigma) \le 2k+ \epsilon_\chi$.
  \item If $\LO_{\psi,\chi} (\sigma) = 2k + \epsilon_\chi < \dim X$, then $L^S (s,\sigma \times \chi^{-1})$ is holomorphic for $\Re (s ) > \half (\dim X+1-\epsilon_\chi) - k$.
  \item If $\LO_{\psi,\chi} (\sigma) = 2k + \epsilon_\chi \ge \dim X $, then $L^S (s,\sigma \times \chi^{-1})$ is holomorphic for $\Re (s ) \ge 1/2$.
  \end{enumerate}
\end{thm}

\section{Period Integrals and Main Results}
\label{sec:peri-integr-main}

The main goal of this paper is to generalise the results for orthogonal groups\cite{MR2540878}, symplectic groups\cite{MR3805648} and metaplectic groups\cite{wu:_period-metaplectic} to unitary groups.  We derive results on periods of residues of Eisenstein series and the first occurrence index, which is much finer than the results on locations of poles and lowest occurrence. One of the consequences is that we can strengthen Thm.~\ref{thm:Eis-pole-LO}, so that part \eqref{item:LO-pole-ineq} becomes an equality. 

We begin by introducing certain period integrals. Let  $\Theta_{\psi,\chi_1,\chi_2,X,Y}$ be the automorphic representation of $G(X)(\AA)$ spanned by the theta functions $\theta_{\psi,\chi_1,\chi_2,X,Y} (\cdot,1,\Phi)$ for $\Phi$
running over $\cS_{X,Y} (\A)$ and  write $\overline{\Theta_{\psi,\chi_1,\chi_2,X,Y}}$ for the complex conjugate representation.
In \cite{MR3435720}, we studied  periods on $\sigma\otimes \overline {\Theta_{\psi,\chi_1,\chi_2,X,Y}}$, namely, period integrals of the form:
\begin{equation}\label{eq:int-sigma-theta}
  \int_{[J]} f(g)\overline{\theta_{\psi,\chi_1,\chi_2,X,Y}(g,1,\Phi)}dg,
\end{equation}
where $J$ is a subgroup of $G(X)$. Most often, $J=G(Z)$ where $Z$ is a non-degenerate skew-Hermitian subspace of $X$. 

We will consider period integrals involving the Eisenstein series (see \eqref{eq:period-eis-theta} and \eqref{eq:period-residue-2}). They bear some resemblance to \eqref{eq:int-sigma-theta} and our computation below will relate them to \eqref{eq:int-sigma-theta}.

First we summarise and extend the results on period integrals \eqref{eq:int-sigma-theta} in Sections~5.1 and 5.3 of \cite{MR3435720}. We rephrase them using our current notation and we  use the language of distinction by subgroups to express if the period integrals vanish or not for some choice of data. As both $G (X)$ and $G (Y)$ are unitary groups, their roles can be exchanged. Thus we write just one direction here.

We define some subgroups of $G (X)$. Let $Z$ be a non-degenerate skew-Hermitian subspace of $X$ and let $L$ be a totally
isotropic subspace of $Z$. Then $G (Z)$ is a unitary subgroup of $G (X)$. Let $Q (Z,L)$ be the parabolic subgroup of
$G (Z)$ that stabilisers $L$, $N(Z,L)$ the unipotent radical of $Q (Z,L)$ and $J (Z,L)$  the subgroup of $Q (Z,L)$ consisting of the elements of  $Q (Z,L)$ that fix $L$ element-wise. We remark that $J (Z,L)\isom N(Z,L)\rtimes G(Z)$ is a Jacobi group when $L\neq 0$. When $L=0$, then $J (Z,L)$ is just $G (Z)$. Sometimes we will write $Q(Z,v)$ as a short hand for $Q(Z,L)$ if $v$ spans $L$. A similar definition works for the Hermitian space $Y$.

Then we have the following on vanishing and non-vanishing of \eqref{eq:int-sigma-theta}.

\begin{prop}\label{prop:period-X->Y}
  Let $\sigma\in\cA_\cusp(G(X))$ and $Y$ a (possibly trivial) anisotropic Hermitian space. Assume that
  \begin{equation*}
    \FO_{\psi,\chi_1,\chi_2,X}^{Y} (\sigma) = \dim Y + 2r_0.
  \end{equation*}
   Then  the following hold.
  \begin{enumerate}
  \item \label{item:1} For any non-degenerate skew-Hermitian subspace $Z$ of $X$, totally isotropic subspace $L$ of $Z$ and non-negative integer  $r$, if $\dim X - \dim Z +\dim L +r <r_0$ then $\sigma\otimes \overline {\Theta_{\psi,\chi_1,\chi_2,X,Y_r}}$ is not $J (Z,L)$-distinguished.
\item \label{item:3} There exists a non-degenerate skew-Hermitian subspace $Z$ of $X$ satisfying $\dim X - \dim Z = r_0$ such that $\sigma\otimes \overline {\Theta_{\psi,\chi_1,\chi_2,X,Y}}$ is $G (Z)$-distinguished.
  \end{enumerate}
\end{prop}
\begin{rmk}
  We will only use $r=0$ in Prop.~\ref{prop:period-X->Y}\eqref{item:1} later.
\end{rmk}
\begin{proof}
  The proof is analogous to that of \cite[Prop.~3.2]{MR3805648} in the symplectic case. We will be brief here.
  The idea is to take Fourier coefficients of $\theta_{\psi,\chi_1,\chi_2,X}^{Y_t}(\sigma)$ for varying $t$. If $t< r_0$, then $\theta_{\psi,\chi_1,\chi_2,X}^{Y_t}(\sigma)$ is an earlier theta lift than the first occurrence and hence its Fourier coefficients are all zero. Via an inductive argument, we get the vanishing result \eqref{item:1}. If $t=r_0$, then $\theta_{\psi,\chi_1,\chi_2,X}^{Y_t}(\sigma)$ is the first occurrence which by \cite{MR1166512} has a non-vanishing non-singular Fourier coefficients. This gives us the non-vanishing result \eqref{item:3}.
\end{proof}

 In the remaining sections, we will  write $G(X)$ for $G(X)(F)$ and similarly for other algebraic groups over $F$ to reduce clutter unless there is confusion. 

 Now we come to the key result of this article.
We study period integrals involving the Eisenstein series:
\begin{equation}
  \int_{[G (Z_1)]}  E (g,f_s) \overline{\theta_{\psi,\chi_1,\chi_2,X_1,Y} (g,h,\Phi)} dg
\end{equation}
with $f_s\in \cA_1(s,\chi_1,\sigma)$. We note that, compared to \eqref{eq:int-sigma-theta}, we have $X_1$ (resp. $Z_1$) instead of $X$ (resp. $Z$).
This integral is not absolutely convergent in general. To make sense of it, we apply the Arthur truncation.
Note that the Eisenstein series is attached to the maximal parabolic subgroup $Q_1$ of $G (X_1)$. In this case, $\fa_{M_1}$ is
$1$-dimensional and we identify it with $\RR$. For $c\in \RR_{>0}$, set $\hat {\tau}^c$ to be the characteristic function
of $\RR_{> \log c}$ and set $\hat {\tau}_c = 1_\RR - \hat {\tau}^c$. Then the truncated Eisenstein series is
\begin{equation}\label{eq:truncation}
  \Lambda^c E (g,f_s) = E (g,f_s) - \sum_{\gamma \in Q_1 \lmod G (X_1)} E_{Q_1} (\gamma g , f_s) \hat {\tau}^c (H  (\gamma g))
\end{equation}
where $E_{Q_1} (\cdot,  f_s)$ is the constant term of $E (\cdot,f_s)$ along $Q_1$ and $H: G(X_1)(\AA) \rightarrow \fa_{M_1}$ is the Harish-Chandra map. More precisely, for $m\in M_1(\AA)$, we have
\begin{equation*}
  \exp(\form{\mu}{H(m)}) = \prod_{v} |\mu(m_v) |_v
\end{equation*}
for all rational characters $\mu $ of $M_1$, where the pairing is between $\fa_{M_1}^*$ and $\fa_{M_1}$ and the product runs over all places $v$ of $F$. Then $H$ is extended to $G(X_1)(\AA)$ via the Iwasawa decomposition.
The summation has only finitely many
non-vanishing terms for each fixed $g$. The truncated Eisenstein series is
rapidly decreasing while the theta series is of moderate growth. Thus if we replace the Eisenstein series with the truncated Eisenstein
series, we get a period integral that is absolutely convergent:
\begin{equation}\label{eq:period-eis-theta}
  \int_{[G (Z_1)]} \Lambda^c E (g,f_s) \overline {\theta_{\psi,\chi_1,\chi_2,X_1,Y} (g,1,\Phi)} dg.
\end{equation}

For $\Re s > \rho_{Q_1} = \half (\dim X +1)$ we have
\begin{equation*}
  E_{Q_1} (g,f_s) = f_s (g) + M (w_0,s)f_s (g),
\end{equation*}
where $w_0$ is  the longest Weyl group element in $Q_1\lmod G (X_1)/Q_1$.
See (2.3), (2.4) and (2.5) in \cite{MR3435720}, where it was denoted by  $w_{\emptyset}$, for detail.  The above identity holds for all $s$ as meromorphic functions. 

Thus the truncated Eisenstein series $\Lambda^c E (g,f_s)$ equals
\begin{equation*}
  \sum_{\gamma \in Q_1 \lmod G (X_1)} f_s (\gamma g) \hat {\tau}_c (H (\gamma g)) 
- \sum_{\gamma \in Q_1 \lmod G (X_1)} M (w_0,s) f_s (\gamma g) \hat {\tau}^c (H (\gamma g)).
\end{equation*}
The first summation is absolutely convergent for $\Re s > \rho_{Q_1}$ and the second has only finitely many terms for fixed $g$. Both have meromorphic continuation to the whole complex plane.

Set 
\begin{align}\label{eq:defn_xi}
  \xi_{c,s} (g) = & \overline{\theta_{\psi,\chi_1,\chi_2,X_1,Y} (g, 1, \Phi)} f_s (g)\hat {\tau}_c (H (g)) ;\nonumber\\
  \xi_s^c (g) = & \overline {\theta_{\psi,\chi_1,\chi_2,X_1,Y} (g, 1, \Phi)} M (w_0,s) f_s (g)\hat {\tau}^c (H (g)) .
\end{align}
 Also for 
 \begin{equation*}
   \xi=\xi_{c,s} \quad \text{or} \quad \xi_s^c,
 \end{equation*}
 we set
\begin{equation}\label{eq:I_xi}
  I (\xi) =\int_{[G (Z_1)]} \sum_{\gamma\in Q_1 \lmod G (X_1)} \xi (\gamma g) dg.
\end{equation}
Thus
\begin{equation*}
   \int_{[G (Z_1)]} \Lambda^c E (g,f_s) \overline {\theta_{\psi,\chi_1,\chi_2,X_1,Y} (g,1,\Phi)} dg  
   = I (\xi_{c,s})  - I (\xi_s^c),   
 \end{equation*}
 as long as the two terms on the right hand side are absolutely convergent.
 We will cut each term on the right hand side into several pieces, show that each piece is absolutely convergent at least when the truncation parameter $c$ is large enough and when $\Re s$ is large enough, and compute their values. The values will be in terms of period integrals on $\sigma\otimes \overline {\Theta_{\psi,\chi_1,\chi_2,X,Y}}$ and of certain explicit functions in $s$. In this way we relate analytic properties of the Eisenstein series to the  vanishing or non-vanishing of period integrals on $\sigma\otimes \overline {\Theta_{\psi,\chi_1,\chi_2,X,Y}}$.

We decompose $ Q_1 \lmod G (X_1)$ according to $G (Z_1)$-orbits and unfold \eqref{eq:I_xi} formally to get
\begin{align*}
  I (\xi) = &\int_{[G (Z_1)]} \sum_{\gamma \in Q_1 \lmod G (X_1)} \xi (\gamma g) dg\\
  = & \sum_{\gamma \in Q_1 \lmod G (X_1) / G (Z_1)} \int_{G (Z_1)^\gamma \lmod G (Z_1) (\A)} \xi (\gamma g) dg,
\end{align*}
where $G (Z_1)^\gamma = \gamma^{-1}Q_1\gamma \cap G (Z_1)$ is the stabiliser in $G (Z_1)$ of $Q_1\gamma$.

We examine more closely the $G (Z_1)$-orbits in the generalised flag variety $Q_1 \lmod G (X_1)$ which classifies the isotropic lines in $X_1$. We use the correspondence given below:
\begin{align*}
  Q_1 \lmod G (X_1) &\longleftrightarrow \{\text {Isotropic lines in } X_1\}\\
  \gamma &\longleftrightarrow \ell_1^-\gamma.
\end{align*}
For an isotropic line in $X_1$  spanned by $x\in X_1$, we will write $\gamma_{Ex}$ or more succinctly $\gamma_x$ for any element in $G (X_1)$ that maps $\ell_1^-$ to $Ex$. We will generally take $\gamma_x$ to be as simple as possible to facilitate computation below.

We divide $I (\xi)$ into several parts. Write  $X_1 = V \oplus Z_1$ with $V$ being the orthogonal complement of $Z_1$ in $X_1$. Given an isotropic vector $x$ in $X_1$ we decompose it accordingly
\begin{equation*}
  x = v + z
\end{equation*}
for $v\in V$ and $z\in Z_1$. There are several possibilities  and we define
\begin{enumerate}
\item $\Omega_{0,1}$ to be the set of isotropic lines such that $v=0$ and $z \neq 0$ is isotropic,
\item $\Omega_{1,0}$ to be the set of isotropic lines such that $v \neq 0$ is isotropic and $z=0$,
\item $\Omega_{1,1}$ to be the set of isotropic lines such that $v\neq 0$ is isotropic and $z\neq 0$ is isotropic and
\item $\Omega_{2,2}$ to be the set of isotropic lines such that $v\neq 0$ is anisotropic, $z\neq 0$ is anisotropic with $\form{v}{v}_{V} = -  \form{z}{z}_{Z_1}$.
\end{enumerate}
Note that each set is closed under the action of $G (Z_1)$. Compared to the cases of the symplectic groups\cite{MR3805648} and of the metaplectic groups\cite{wu:_period-metaplectic}, we have an extra type $\Omega_{2,2}$ to consider. Thus we have
\begin{equation}\label{eq:cut-into-pieces}
  I (\xi) = I_{\Omega_{0,1}} (\xi) + I_{\Omega_{1,0}} (\xi) +I_{\Omega_{1,1}} (\xi)+I_{\Omega_{2,2}} (\xi),
\end{equation}
where
\begin{equation}\label{eq:I-omega}
  I_{\Omega_{*,*}} (\xi) = \sum_{x\in\Omega_{*,*}  / G (Z_1)} \int_{G (Z_1)^{\gamma_x} \lmod G (Z_1) (\A)}  \xi (\gamma_x g) dg.
\end{equation}
The above is justified once we show absolute convergence for each $I_{\Omega_{*,*}} (\xi)$.

Section~\ref{sec:comp-integrals} is dedicated to computing the values of each $I_{\Omega_{*,*}} (\xi)$. We will see that only $I_{\Omega_{0,1}} (\xi)$ can possibly be non-zero. In Section~\ref{sec:absol-conv-integr}, we will show that every $I_{\Omega_{*,*}} (\xi)$ is absolutely convergent.
Putting these together, we get results on poles and residues of Eisenstein series. The conditions on distinction and non-distinction may look technical in the theorem below, but  such a $Z$ exists  if we assume $\FO_{\psi,\chi}^Y(\sigma) = \dim Y + 2r$ by Prop.~\ref{prop:period-X->Y}.

\begin{thm} \label{thm:dist->Eis-pole}
  Let $\sigma \in \cA_\cusp (G (X))$ and let $Y$ be a
  (possibly trivial) anisotropic Hermitian space. Let $Z$ be a non-degenerate skew-Hermitian subspace of $X$. Set $r=\dim X - \dim Z$ and $s_0=\half (\dim X -( \dim Y +2r) +1)$.
 Assume that $\sigma\otimes \overline {\Theta_{\psi,\chi,\chi_2,X,Y}}$ is  $G(Z)$-distinguished and that it is not $J(Z',L')$-distinguished for all non-degenerate skew-Hermitian subspaces $Z'$ of $X$ and isotropic subspaces $L'$ of $Z'$ such that $\dim L'=0,1$ and $\dim Z' -  \dim L' > \dim Z$.
 Then
$s_0$ is a  pole of $E(g,f_s)$ for some $f_s\in \cA_1 (s,\chi,\sigma)$.
\end{thm}
\begin{proof}
By Propositions~\ref{prop:val-I-Omega10}, \ref{prop:val-I-Omega11} and \ref{prop:val-I-Omega22}, we get vanishing of $I_{\Omega_{1,0}} (\xi)$, $I_{\Omega_{1,1}} (\xi)$ and $I_{\Omega_{2,2}} (\xi)$ respectively where $\xi=\xi_{c,s}$ or $\xi_s^c$.
   Thus
\begin{equation*}
  \int_{[G (Z_1)]} \Lambda^c E (g,f_s) \overline {\theta_{\psi,\chi,\chi_2,X_1,Y} (g,1,\Phi)} dg = I_{\Omega_{0,1}}(\xi_{c,s}) - I_{\Omega_{0,1}}(\xi_s^c).
\end{equation*}

If $M(w_0,s)$ has a pole at $s=s_0$, then as the analytic behaviour of the Eisenstein series is governed by its constant terms\cite[Lemma I.4.10]{MR1361168}, the Eisenstein series has a pole at $s=s_0$.
  
Now assume that $M(w_0,s)$ does not have a pole at $s=s_0$. Then \eqref{eq:int-omega01-intertwine-part-over-t} and \eqref{eq:int-omega01-intertwine-part-part-2} show that $I_{\Omega_{0,1}}(\xi_s^c)$ does not have a pole at $s=s_0$. Prop.~\ref{prop:nonvan-I-omega01} shows that $I_{\Omega_{0,1}}(\xi_{c,s})$ has a pole at $s=s_0$ for some choice of data. Hence the truncated Eisenstein series must have a pole at $s=s_0$ and so must the Eisenstein series. Therefore $M(w_0,s)$ has a pole at $s=s_0$. We get a contradiction, so this case cannot occur.
\end{proof}

For  $s_1>0$, let $\cE_{-1} (g,f_{s_1})$ denote the residue of $E(g,f_s)$ at $s=s_1$. We note that by the general theory of Eisenstein series \cite[IV.1.11]{MR1361168}, the pole is at most simple.
As we will see, in the proof of the following theorem, the  period integral 
\begin{equation}\label{eq:period-residue-2}
  \int_{[G(Z_1)]} \cE_{-1}(g,f_{s_0}) \overline {\theta_{\psi,\chi,\chi_2,X_1,Y} (g,1,\Phi)} dg
\end{equation}
is absolutely convergent. Thus we have the notion of distinction for the automorphic representation $\cE_{-1}(s_0,\chi,\sigma)\otimes \overline{\Theta_{\psi,\chi,\chi_2,X_1,Y}}$ where $\cE_{-1}(s_0,\chi,\sigma)$ is the residue representation that is the span of all $\cE_{-1}(g,f_{s_0})$.

With one additional assumption to Theorem~\ref{thm:dist->Eis-pole}, we get the following theorem.
\begin{thm} \label{thm:dist->period-residue-Eis}
  Let $\sigma \in \cA_\cusp (G (X))$ and let $Y$ be a
  (possibly trivial) anisotropic Hermitian space. Let $Z$ be a non-degenerate skew-Hermitian subspace of $X$. Set $r=\dim X - \dim Z$ and $s_0=\half (\dim X -( \dim Y +2r) +1)$.
 Assume that $\sigma\otimes \overline {\Theta_{\psi,\chi,\chi_2,X,Y}}$ is  $G(Z)$-distinguished and that it is not $J(Z',L')$-distinguished for all non-degenerate skew-Hermitian subspaces $Z'$ of $X$ and isotropic subspaces $L'$ of $Z'$ such that $\dim L'=0,1$ and $\dim Z' -  \dim L' > \dim Z$. Assume further that $s_0 >0$. Then the following hold.
  \begin{enumerate}
  \item \label{item:4}$s_0$ is in $\cP_1 (\chi,\sigma)$ and it is a
simple pole. 
  \item \label{item:5}$\cE_{-1}(s_0,\chi,\sigma)\otimes \overline{\Theta_{\psi,\chi,\chi_2,X_1,Y}}$ is $G(Z_1)$-distinguished.
    \item \label{item:6}$\cE_{-1}(s_1,\chi,\sigma)\otimes \overline{\Theta_{\psi,\chi,\chi_2,X_1,Y}}$
     is not $G(Z_1)$-distinguished for  any $0<s_1\neq s_0$.
   \item \label{item:7}For any $Z''$ with $\dim Z'' > \dim Z$, $\cE_{-1}(s_1,\chi,\sigma)\otimes \overline{\Theta_{\psi,\chi,\chi_2,X_1,Y}}$ is not $G(Z''_1)$-distinguished for  any $s_1>0$.
\end{enumerate}
\end{thm}
\begin{proof}
  We have seen that $s_0$ is a pole for the Eisenstein series in Thm.~\ref{thm:dist->Eis-pole} and since we assume $s_0>0$, it is simple. This proves part \eqref{item:4}. Next we study the residue.
  Consider the truncated residue
  \begin{equation*}
    \Lambda^c \cE_{-1} (g,f_{s_1}) = \cE_{-1} (g,f_{s_1}) - \sum_{\gamma \in Q_1 \lmod G (X_1)}
\cE_{-1,Q_1} (\gamma g,f_{s_1})     \hat {\tau}^c (H
  (\gamma g)).
\end{equation*}
Set
\begin{align*}
  \theta^c(g):=&\sum_{\gamma \in Q_1 \lmod G (X_1)}\cE_{-1,Q_1} (\gamma g,f_{s_1})     \hat {\tau}^c (H  (\gamma g))\\
  =& \sum_{\gamma \in Q_1 \lmod G (X_1)} \res_{s=s_1} M(w_0,s)f_s(\gamma g)\hat {\tau}^c (H
     (\gamma g)).
\end{align*}
Again we  show absolute convergence by cutting up the integrals into several parts. Then by Lemmas~\ref{lemma:conv-01}, \ref{lemma:conv-10}, \ref{lemma:conv-11}, \ref{lemma:conv-22}, we see that
\begin{equation}\label{eq:I-omega-trunc-residue}
  \int_{[G(Z_1)]} \theta^c(g) \overline {\theta_{\psi,\chi,\chi_2,X_1,Y} (g,1,\Phi)} dg
\end{equation}
is absolutely convergent when $c$ is large enough and $s_1>0$. Note that as we assume that $s_0>0$, the quantity $\frac{c^{-s-s_0}}{s+s_0}$ in \eqref{eq:int-omega01-intertwine-part-over-t} will not produce a residue at $s=s_1>0$. Hence \eqref{eq:I-omega-trunc-residue} is equal to $\res_{s=s_1} I(\xi_s^c)$. Thus the period integral of the residue twisted by the theta series is 
\begin{align}\label{eq:period-residue}
  &    \int_{[G(Z_1)]} \cE_{-1}(g,f_{s_1}) \overline {\theta_{\psi,\chi,\chi_2,X_1,Y} (g,1,\Phi)} dg \\
  =&   \int_{[G(Z_1)]} (\Lambda^c\cE_{-1}(g,f_{s_1})+\theta^c(g)) \overline {\theta_{\psi,\chi,\chi_2,X_1,Y} (g,1,\Phi)} dg \nonumber\\
  =& \res_{s=s_1} (I(\xi_{c,s}) - I(\xi_s^c)) + \res_{s=s_1} I(\xi_s^c) \nonumber \\
     =& \res_{s=s_1} I(\xi_{c,s}) \nonumber. 
\end{align}
The first equality shows in addition that \eqref{eq:period-residue} is absolutely convergent.

If $s_1\neq s_0$, then \eqref{eq:period-residue}  vanishes. This proves part \eqref{item:6}. If we change the domain of integration to $[G(Z''_1)]$ with $\dim Z''>\dim Z$, then   Prop.~\ref{prop:nonvan-I-omega01}\eqref{item:11} shows that \eqref{eq:period-residue}  vanishes and we prove part \eqref{item:7}. Now assume $s_1=s_0$. Then Prop.~\ref{prop:nonvan-I-omega01}\eqref{item:12} shows that \eqref{eq:period-residue}  does not vanish for some choice of data and we prove part \eqref{item:5}.
\end{proof}

By Prop.~\ref{prop:period-X->Y}, when $\FO_{\psi,\chi}^Y(\sigma) = \dim Y +2r$, we have such a $Z$ so that the assumptions of Thms.~\ref{thm:dist->Eis-pole}, \ref{thm:dist->period-residue-Eis} are satisfied.
 Thus we get
\begin{cor} \label{cor:FO->Eis-pole}
  Let $\sigma \in \cA_\cusp (G (X))$ and let $Y$ be a
  (possibly trivial) anisotropic Hermitian space. Assume $\FO_{\psi,\chi}^Y(\sigma) = \dim Y +2r$.
  Set  $s_0=\half (\dim X -( \dim Y +2r) +1)$.
 Then $s_0$ is a  pole of $E(g,f_s)$ for some $f_s\in \cA_1 (s,\chi,\sigma)$.
\end{cor}

\begin{cor} \label{cor:FO->period-residue-Eis}
   Let $\sigma \in \cA_\cusp (G (X))$ and let $Y$ be a
  (possibly trivial) anisotropic Hermitian space. Assume $\FO_{\psi,\chi}^Y(\sigma) = \dim Y +2r$.
  Set  $s_0=\half (\dim X -( \dim Y +2r) +1)$. Assume further that $s_0 >0$. Then the following hold.
  \begin{enumerate}
  \item $s_0$ is in $\cP_1 (\chi,\sigma)$ and it is a
simple pole. 
  \item \label{item:2}$\cE_{-1}(s_0,\chi,\sigma)\otimes \overline{\Theta_{\psi,\chi,\chi_2,X_1,Y}}$ is $G(Z_1)$-distinguished for some non-degenerate skew-Hermitian subspace $Z$ of $X$ such that $\dim X - \dim Z = r$.
    \item $\cE_{-1}(s_1,\chi,\sigma)\otimes \overline{\Theta_{\psi,\chi,\chi_2,X_1,Y}}$
     is not $G(Z_1)$-distinguished for  any $0<s_1\neq s_0$ and for any non-degenerate skew-Hermitian subspace $Z$ of $X$.
   \item For any non-degenerate skew-Hermitian subspace $Z''$ of $X$ with $\dim Z'' > \dim X -r$, $\cE_{-1}(s_1,\chi,\sigma)\otimes \overline{\Theta_{\psi,\chi,\chi_2,X_1,Y}}$ is not $G(Z''_1)$-distinguished for  any $s_1>0$. 
\end{enumerate}
\end{cor}

Using Cor.~\ref{cor:FO->Eis-pole}, we can now strengthen Part~\eqref{item:LO-pole-ineq} of Thm.~\ref{thm:Eis-pole-LO} to an equality.
\begin{cor}\label{cor:Eis-pole=LO}
  Let  $\sigma\in\cA_\cusp(G(X))$. Then  $s_0$ is the maximal pole of the Eisenstein series $E(g,f_s)$ for  $f_s$ running over $\cA_1(s,\chi,\sigma)$ if and only if $\LO_{\psi,\chi} (\sigma) =  \dim X + 1 - 2 s_0$.
\end{cor}
\begin{proof}
  Assume that $s_0$ is the maximal pole of the Eisenstein series. By Thm.~\ref{thm:Eis-pole-LO}, we already know $\LO_{\psi,\chi} (\sigma) \le  \dim X + 1 - 2 s_0$. Thus $\LO_{\psi,\chi} (\sigma) = \dim X + 1 - 2 s_1$ for some $s_1 \ge s_0$. The lowest occurrence is a first occurrence in a certain Witt tower. Thus by Cor.~\ref{cor:FO->Eis-pole}, $\half(\dim X - (\dim X + 1 - 2 s_1) +1) = s_1$ is a pole of the Eisenstein series. As $s_0$ is assumed to be maximal, we are forced to have $s_1=s_0$, which means that $\LO_{\psi,\chi} (\sigma) =  \dim X + 1 - 2 s_0$ as desired.

  Next assume that $\LO_{\psi,\chi} (\sigma) =  \dim X + 1 - 2 s_0$ for some $s_0$. Again by Cor.~\ref{cor:FO->Eis-pole}, $\half(\dim X - (\dim X + 1 - 2 s_0) +1) = s_0$ is a pole of the Eisenstein series. Assume that $s_1$ is the maximal pole of the Eisenstein series. By what we have shown, $\LO_{\psi,\chi} (\sigma) =  \dim X + 1 - 2 s_1$. Thus $s_1=s_0$; in other words, $s_0$ is the maximal pole of the Eisenstein series.
\end{proof}


\section{Computation of Values of Integrals}
\label{sec:comp-integrals}
We have cut the period integral of the truncated Eisenstein series twisted by a theta function into pieces \eqref{eq:cut-into-pieces} and we proceed to compute their values assuming absolute convergence which is treated in Sec.~\ref{sec:absol-conv-integr}.

\subsection{The value of $I_{\Omega_{0,1}} (\xi)$}
\label{sec:series-over-omega_01}

 The set $\Omega_{0,1}$ consists of one $G (Z_1)$-orbit. We take the representative $1$ which corresponds to the isotropic line $\ell_1^- = Ee_1^-$ in $X_1$ or even in $Z_1$. The stabiliser of $Ee_1^-$ in $G (Z_1)$ is the parabolic subgroup $Q (Z_1,Ee_1^-)$.

Thus we find
\begin{align*}
  I_{\Omega_{0,1}} (\xi)=&\sum_{x\in\Omega_{0,1}  / G (Z_1)} \int_{G (Z_1)^{\gamma_x}(F) \lmod G (Z_1) (\A)}  \xi (\gamma_x g) dg \\
  = &\int_{Q(Z_1,e_1^-)(F)\lmod G(Z_1)(\AA)} \xi (g)dg.
\end{align*}
Using the Iwasawa decomposition, we write $g=nm_1(t)hk$ for $n\in N(Z_1,e_1^-)(\AA)$, $t\in \RGL_1(\AA)$, $h\in G(Z)(\AA)$ and $k\in K_{G(Z_1)}$.
Thus, by \cite[I.1.13]{MR1361168}, the integral above is equal to
\begin{align*}
  \int_{K_{G(Z_1)}}  \int_{[G(Z)]} \int_{[\RGL_1]}\int_{[N(Z_1,e_1^-)]} \xi(nm_1(t)hk) |t|^{-2\rho_{Q(Z_1,e_1^-)}} dn dt dh dk.
\end{align*}

First let $\xi=\xi_{c,s}$. We get
\begin{align}
  &\int_{K_{G(Z_1)}}  \int_{[G(Z)]} \int_{[\RGL_1]}\int_{[N(Z_1,e_1^-)]} \overline{\theta_{\psi,\chi_1,\chi_2,X_1,Y} (nm_1(t)hk, 1, \Phi)} \nonumber\\
  & \quad f_s (nm_1(t)hk)\hat {\tau}_c (H (nm_1(t)hk)) |t|^{-2\rho_{Q(Z_1,e_1^-)}} dn dt dh dk \nonumber\\
  \label{eq:int-omega01-1}
  =& \int_{K_{G(Z_1)}}  \int_{[G(Z)]} \int_{[\RGL_1]}\int_{[N(Z_1,e_1^-)]} \overline{\theta_{\psi,\chi_1,\chi_2,X_1,Y} (nm_1(t)hk, 1, \Phi)} \nonumber\\
  & \quad \chi_1(t) |t|^{s+\rho_{Q_1}} 
   f_s (hk)\hat {\tau}_c (H (m_1(t))) |t|^{-2\rho_{Q(Z_1,e_1^-)}} dn dt dh dk.
\end{align}
Only the theta term depends on $n$. We  compute the inner integral over $dn$.
Note that $N(Z_1,e_1^-)$ is isomorphic to $\Hom_E (\ell_1^+, Z) \rtimes \Herm_1$. See, for example, \cite[(2.16)]{MR3071813}. We parametrise
  $n\in N(Z_1,e_1^-) $   as $n(\mu,\beta)$ for $\mu\in \Hom_E (\ell_1^+, Z)$ and $\beta\in \Herm_1$. We use the mixed model for the Weil representation. Then we get
\begin{align*}
  &\int_{[N(Z_1,e_1^-)]} \theta_{\psi,\chi_1,\chi_2,X_1,Y} (nm_1(t)hk, 1, \Phi)  dn \\
   =&\int_{[\Hom_E (\ell_1^+, Z)]}\int_{[\Herm_1]} \sum_{y\in Y}\sum_{w\in R_{E/F} (Y\otimes X)^+}\omega_{\psi,\chi_1,\chi_2,X_1,Y} (n (\mu,\beta) m_1(t)hk,1)\Phi (y,w)  d\beta d\mu  \\
  =&\int_{[\Hom_E (\ell_1^+, X)]}\int_{[\Herm_1]} \sum_{y\in Y}\sum_{w\in R_{E/F} (Y\otimes X)^+}\omega_{\psi,\chi_1,\chi_2,X_1,Y} (n (\mu,0) m_1(t)h k,1)\Phi (y,w) \\
  & \quad \psi_E (\half \form{y}{y}_{Y}\beta) d\beta d\mu  .
\end{align*}
The integration over $\beta$ vanishes unless $y=0$ since $Y$ is anisotropic. We continue the computation:
\begin{align*}
  &\int_{[\Hom_E (\ell_1^+, Z)]}\sum_{w\in R_{E/F} (Y\otimes X)^+}\omega_{\psi,\chi_1,\chi_2,X_1,Y} (n (\mu,0) m (t,h) k,1)\Phi (0,w)d\mu  \\
=&\sum_{w\in R_{E/F} (Y\otimes X)^+}\omega_{\psi,\chi_1,\chi_2,X_1,Y} ( m_1 (t)h k,1)\Phi (0,w)
\end{align*}
by (2.19) of \cite{MR3071813}.
Writing out the explicit action of $t$,  we get
\begin{equation*}
  \chi_1(t)|t|^{\dim Y/2}\theta_{\psi,\chi_1,\chi_2,X,Y} ( h, 1, \Phi'_k)
\end{equation*}
where $\Phi'_k (\cdot)= \omega_{\psi,\chi_1,\chi_2,X_1,Y} (k,1)\Phi (0,\cdot)$ lies in $\cS (R_{E/F} (Y\otimes X)^+(\AA))$.
Plugging  the above into \eqref{eq:int-omega01-1}, we then consider the inner integral over $t$ and this gives
\begin{align*}
 \int_{[\RGL_1]} |t|^{\half\dim Y + s +\rho_{Q_1}-2\rho_{Q(Z_1,e_1^-)}} \hat {\tau}_c (H (m_1 (t)))  d t
= \vol ( E^\times \lmod \A_E^1) \int_{0}^c t^{s-s_0}  d^\times t
\end{align*}
where  we set $r=\dim X - \dim Z$ and $s_0=\half (\dim X -( \dim Y +2r) +1)$. Assume $\Re s > s_0$. Then the above is absolutely convergent and is equal to
\begin{align}\label{eq:int-omega01-over-t}
   \vol ( E^\times \lmod \A_E^1) \frac{c^{s-s_0}}{s-s_0}.
\end{align}
Hence $I_{\Omega_{0,1}}(\xi_{c,s})$ is equal to \eqref{eq:int-omega01-over-t} times
\begin{align}\label{eq:int-omega01-part-2}
  \int_{K_{G(Z_1)}} \int_{[G (Z)]} \overline{\theta_{\psi,\chi_1,\chi_2,X,Y} ( h, 1, \Phi'_k)} f_s (h k) dh dk.
\end{align}
This expression gives the meromorphic continuation of $I_{\Omega_{0,1}}(\xi_{c,s})$ to the whole complex plane.

Next let $\xi=\xi_s^c$. An analogous computation shows that if $\Re s > -s_0$ then
$I_{\Omega_{0,1}}(\xi_s^c)$ is equal to the product of 
\begin{align}
&  \int_{[\RGL_1]} |t|^{\half\dim Y - s +\rho_{Q_1}-2\rho_{Q(Z_1,e_1^-)}} \hat {\tau}^c (H (m_1 (t)))  d t \nonumber \\
  =& \vol ( E^\times \lmod \A_E^1) \int_{c}^\infty t^{-s-s_0}  d^\times t \nonumber \\
     =& \vol ( E^\times \lmod \A_E^1) \frac{c^{-s-s_0}}{s+s_0} \label{eq:int-omega01-intertwine-part-over-t}
\end{align}
and
\begin{equation}\label{eq:int-omega01-intertwine-part-part-2}
  \int_{K_{G(Z_1)}} \int_{[G (Z)]} \overline{\theta_{\psi,\chi_1,\chi_2,X,Y} ( h, 1, \Phi'_k)} M(w_0,s)f_s (h k) dh dk.
\end{equation}
This expression gives the meromorphic continuation of $I_{\Omega_{0,1}}(\xi_s^c)$ to the whole complex plane.

To summarise, we have
\begin{prop} \label{prop:nonvan-I-omega01}
  Let $Z$ be a non-degenerate skew-Hermitian subspace of $X$. Set $r=\dim X - \dim Z$ and $s_0=\half (\dim X -( \dim Y +2r) +1)$. Let $\xi_{c,s}$ and $\xi_c^s$ be as in \eqref{eq:defn_xi}. Then the following holds
  \begin{enumerate}
  \item \label{item:11}If $\sigma\otimes\overline{\Theta_{\psi,\chi_1,\chi_2,X,Y}}$ is not $G(Z)$-distinguished, then both $I_{\Omega_{0,1}}(\xi_{c,s})$ and $I_{\Omega_{0,1}}(\xi_s^c)$ vanish.
  \item $I_{\Omega_{0,1}}(\xi_{c,s})$ does not have a pole at $s\neq s_0$.
  \item \label{item:12}If $\sigma\otimes\overline{\Theta_{\psi,\chi_1,\chi_2,X,Y}}$ is  $G(Z)$-distinguished, then the residue of $I_{\Omega_{0,1}}(\xi_{c,s})$ at $s=s_0$ is equal to
    \begin{equation}\label{eq:residue}
      \vol ( E^\times \lmod \A_E^1)\int_{K_{G(Z_1)}} \int_{[G (Z)]} \overline{\theta_{\psi,\chi_1,\chi_2,X,Y} ( h, 1, \Phi'_k)} f_{s_0} (h k) dh dk,
    \end{equation}
    where $\Phi'_k (\cdot)= \omega_{\psi,\chi_1,\chi_2,X_1,Y} (k,1)\Phi (0,\cdot)$ and for some choice of data, namely, $\Phi\in\cS_{X_1,Y}(\AA)$ and $f_s\in\cA_1(s,\chi_1,\sigma)$, the residue is non-vanishing.
  \end{enumerate}
\end{prop}
\begin{rmk}
  Only the part where the residue is non-vanishing for some choice of data is not clear from the discussion above. Since we assume that $\sigma\otimes\overline{\Theta_{\psi,\chi_1,\chi_2,X,Y}}$ is  $G(Z)$-distinguished, there exist $\phi\in\sigma$ and $\Psi\in\cS_{X,Y}(\AA)$ such that
  \begin{equation*}
    \int_{[G(Z)]} \phi(h)\overline{\Theta_{\psi,\chi_1,\chi_2,X,Y}(h,1,\Psi)} dh
  \end{equation*}
  is non-vanishing. We can extend $\phi$ to a section $f_s\in\cA_1(s,\chi_1,\sigma)$ and $\Psi\in\cS_{X,Y}(\AA)$ to some $\Phi\in\cS_{X_1,Y}(\AA)$ in an appropriate fashion so that \eqref{eq:residue} is non-vanishing. This argument was carried out in great details in \cite[Prop.~4.1]{MR3805648} for the symplectic case and can be easily adapted to the unitary case. Thus we do not reproduce it here.
\end{rmk}

\subsection{The value of $I_{\Omega_{1,0}} (\xi)$}
\label{sec:series-over-omega_10}

Since $G (Z_1)$ acts trivially on $\Omega_{1,0}$, the orbit representatives are just the isotropic lines in $V$, which we parametrise using a set of representatives of $E^\times \lmod (V - \{0\})$ that are isotropic. For each isotropic $v\in V - \{0\}$, we fix a dual vector $v^+\in V$ such that $\form{v^+}{v} = 1$. Then we choose $\gamma_v \in G(X_1)$ to be the element with the action $e_1^+ \leftrightarrow v^+$  and $e_1^- \leftrightarrow v$ and the identity action on the orthogonal complement. Thus
\begin{align}\label{eq:I_omega10}
  I_{\Omega_{1,0}} (\xi) &= \sum_{\substack{v \in E^\times \lmod (V - \{0\}) \\\text {isotropic}}} \int_{[G (Z_1)]}  \xi (\gamma_v g) dg \nonumber \\
  =& \sum_{\substack{v \in E^\times \lmod (V - \{0\}) \\\text {isotropic}}} \int_{ [ G (Ev^+ \oplus Z \oplus Ev)] }  \xi (g\gamma_v) dg.
\end{align}

Let $\xi=\xi_{c,s}$.  The  restriction of the function $\xi (\cdot\ \gamma_v)$ on $G(X_1)(\AA)$ to $G(X)(\AA)$ is equal to
\begin{equation}\label{eq:restrict-xi-to-GX}
  g\in G(X)(\AA) \mapsto \overline{\theta_{\psi,\chi_1,\chi_2,X_1,Y} (g, 1, \omega_{\psi,\chi_1,\chi_2,X_1,Y}(\gamma_v)\Phi)} f_s (g\gamma_v).
\end{equation}
The theta function lies in $\overline{\Theta_{\psi,\chi_1,\chi_2,X,Y}}$ and  $f_s(\cdot\ \gamma_v)$ lies in the space of $\sigma$ for each fixed $\gamma_v$.
Thus the  integral in \eqref{eq:I_omega10} is a period integral for $\sigma\otimes \overline{\Theta_{\psi,\chi_1,\chi_2,X,Y}}$ over $[G (Ev^+ \oplus Z \oplus Ev)] $. The case of $\xi=\xi_s^c$ is similar. Thus we have shown:
\begin{prop}\label{prop:val-I-Omega10}
  If $\sigma\otimes \overline{\Theta_{\psi,\chi_1,\chi_2,X,Y}}$ is not $G(Z')$-distinguished for all non-degenerate skew-Hermitian subspace $Z'$ of $X$ with $Z' \supset Z$ and $\dim Z' = \dim Z +2$, then $I_{\Omega_{1,0}} (\xi)$ vanishes for $\xi = \xi_{c,s}$ and $\xi_s^c$. 
\end{prop}

\subsection{The value of $I_{\Omega_{1,1}} (\xi)$}
\label{sec:series-over-omega_11}

Now consider the action of $G (Z_1)$ on $\Omega_{1,1}$. An isotropic line $E (v+z)$ in $\Omega_{1,1}$ is in the same orbit as $E (v+e_1^-)$. Next we check when $E (v_1+e_1^-)$ and $E (v_2+e_1^-)$ are in the same orbit. This means that there exists $\gamma\in G (Z_1)$ such that $(v_1+e_1^-)\gamma = c (v_2+e_1^-)$ for some $c\in E^\times$. This gives $v_1=cv_2$ and $e_1^-\gamma = c e_1^-$. Thus a set of representatives for the $G (Z_1)$-orbits on $\Omega_{1,1}$ is $E (v+e_1^-)$ where $v$ runs over a set of representatives of $(V-\{0\})/E^\times$ that are isotropic. The stabiliser of $E (v+e_1^-)$ is $J (Z_1,Ee_1^-)$.

We see that
\begin{align}\label{eq:I_omege11}
  I_{\Omega_{1,1}} (\xi) =& \sum_{\substack{v \in E^\times \lmod (V - \{0\}) \\\text {isotropic}}}   \int_{J(Z_1,e_1^-)(F) \lmod G (Z_1) (\A)}  \xi (\gamma_{v+e_1^-} g) dg \nonumber \\
  =& \sum_{\substack{v \in E^\times \lmod (V - \{0\}) \\\text {isotropic}}} \int_{K_{G(Z_1)}} \int_{\RGL_1(\AA)}  \int_{[J(Z_1,e_1^-)]}    \xi (\gamma_{v+e_1^-} g m_1(t)k)  \nonumber \\ &|t|^{-2\rho_{Q(Z_1,e_1^-)}}dg dt dk .
\end{align}

For an isotropic $v\in V - \{0\}$, choose an isotropic  $v^+\in V$ such that $\form{v^+}{v}=1$. We may choose $\gamma_{v+e_1^-}$ to be the element in $G(X_1)$ determined by
\begin{align*}
  e_1^+ &\mapsto v^+ \\
  v^+ & \mapsto e_1^+ - v^+ \\
  v &\mapsto e_1^- \\
  e_1^- &\mapsto v+e_1^-
\end{align*}
with the identity action on the orthogonal complement. With respect to the basis $e_1^+,v^+,v,e_1^-$, we have
\begin{equation*}
  \gamma_{v+e_1^-} =
  \left(
  \begin{array}{cc|cc} 0&1&&\\ 1&-1&&\\ \hline &&&1\\ &&1&1
  \end{array} \right)
  =\left(
  \begin{array}{cc|cc} &1&&\\ 1&&&\\ \hline &&&1\\ &&1&
  \end{array} \right) \left( \begin{array}{cc|cc}
1&-1&&\\ &1&&\\ \hline &&1&1\\ &&&1
   \end{array} \right) =: \gamma_1\gamma_2.
  \end{equation*}
  It is clear that $\gamma_2$ commutes with $J(Z_1,e_1^-)(\AA)$ and $\gamma_1J(Z_1,e_1^-)(\AA) \gamma_1^{-1} = J(Ev^+\oplus Z \oplus Ev, v)(\AA)$. Thus \eqref{eq:I_omege11} is equal to
  \begin{align*}
    \sum_{\substack{v \in E^\times \lmod (V - \{0\}) \\\text {isotropic}}} \int_{K_{G(Z_1)}} \int_{\RGL_1(\AA)}  \int_{[J(Ev^+\oplus Z \oplus Ev, v)]}    \xi ( g \gamma_{v+e_1^-} m_1(t)k) |t|^{-2\rho_{Q(Z_1,e_1^-)}}dg dt dk.
  \end{align*}
For $\xi=\xi_{c,s}$ or $\xi_s^c$ as in \eqref{eq:defn_xi},   by a similar reasoning as in the paragraph that contains \eqref{eq:restrict-xi-to-GX}, we see that the innermost  integral is a period integral of $\sigma\otimes\overline{\Theta_{\psi,\chi_1,\chi_2,X,Y}}$ over $[J(Ev^+\oplus Z \oplus Ev, v)]$. Thus we get:
\begin{prop}\label{prop:val-I-Omega11}
    If $\sigma\otimes \overline{\Theta_{\psi,\chi_1,\chi_2,X,Y}}$ is not $J(Z',L')$-distinguished for all non-degenerate skew-Hermitian subspace $Z'$ of $X$ such that $Z' \supset Z$ and $\dim Z' = \dim Z +2$ and $L'$ is an isotropic line in the orthogonal complement of $Z$ in $Z'$, then $I_{\Omega_{1,1}} (\xi)$ vanishes for $\xi = \xi_{c,s}$ and $\xi_s^c$. 
\end{prop}
  
\subsection{The value of $I_{\Omega_{2,2}} (\xi)$}
\label{sec:series-over-omega_22}

Finally consider the action of $G (Z_1)$ on $\Omega_{2,2}$. Recall that $\daleth$ is a traceless element  in $E^\times$ that we fixed in Sec.~\ref{sec:notation-preliminary}. Let $\alpha$ run through a set of representatives for $\daleth F^\times/\Nm_{E/F}E^\times$. For each $\alpha$, let $z_\alpha= e_1^- - \frac{1}{2}\alpha e_1^+$. Then $\form{z_\alpha}{z_\alpha}_{Z_1} = -\alpha$. Then a set of representatives for the $G (Z_1)$-orbits on $\Omega_{2,2}$ consists of $E (v+z_\alpha)$ where $v\in V$ with $\form{v}{v}_V=\alpha$ and $\alpha$ running through a set of representatives for $\daleth F^\times/\Nm_{E/F}E^\times$. We find the stabiliser of $E (v+z_\alpha)$. Suppose $\gamma\in G (Z_1)$ stabilises  $E (v+z_\alpha)$. This means $(v+z_\alpha)\gamma = c (v+z_\alpha)$ for some $c\in E^\times$. Thus $v=cv$ and $z_\alpha\gamma = cz_\alpha$. Thus $c=1$ and $\gamma$ fixes $z_\alpha$. In other words, $\gamma \in G (z_\alpha^{\perp} \cap Z_1)$. Let $z_\alpha' = e_1^- + \frac{1}{2}\alpha e_1^+\in Z_1$. It is orthogonal to $z_\alpha$. Then we can write $G (z_\alpha^{\perp} \cap Z_1)$ as $G (Z\oplus Ez_\alpha')$. Thus the integral we are concerned with is
\begin{align}
  I_{\Omega_{2,2}} (\xi) &= \sum_{x\in\Omega_{2,2}  / G (Z_1)} \int_{G (Z_1)^{\gamma_x} \lmod G (Z_1) (\A)}  \xi (\gamma_x g) dg\nonumber\\
                         &=\sum_{\alpha\in\daleth F^\times/\Nm_{E/F}E^\times} \sum_{v\in V, \form{v}{v}=\alpha} \int_{ G (Z\oplus Ez_\alpha') (F)\lmod G (Z_1) (\A)} \xi (\gamma_{v+z_\alpha} g) dg \nonumber\\
  &=\sum_{\alpha\in\daleth F^\times/\Nm_{E/F}E^\times} \sum_{v\in V, \form{v}{v}=\alpha}\int_{ G (Z\oplus Ez_\alpha') (\AA)\lmod G (Z_1) (\AA)} \int_{ [G (Z\oplus Ez_\alpha')]} \nonumber \\ \label{eq:I-Omega-22} & \qquad\xi (\gamma_{v+z_\alpha} hg) dh dg.
\end{align}
We choose $\gamma_{v+z_\alpha}$ determined by
\begin{align*}
e_1^- &\mapsto v+z_\alpha \\
 e_1^+ &\mapsto \frac{1}{2\alpha} ( v-z_\alpha) \\
v &\mapsto z_\alpha' := e_1^- + \half \alpha e_1^+
\end{align*}
with the identity action on the orthogonal complement.
 We note that $\form{z_\alpha'}{z_\alpha'}=\alpha$ and $\form{z_\alpha}{z_\alpha'}=0$. It is clear that $\gamma_{v+z_\alpha}G (Z\oplus Ez_\alpha') \gamma_{v+z_\alpha}^{-1} = G(Z\oplus Ev)$. Thus the inner integral is
\begin{align*}
  \int_{ [G (Z\oplus Ev)]} \xi (h \gamma_{v+z_\alpha}g) dh,
\end{align*}
which is a period integral of $\sigma\otimes\overline{\Theta_{\psi,\chi_1,\chi_2,X,Y}}$ over $[G (Z\oplus Ev)]$, by a similar reasoning as in the paragraph that contains \eqref{eq:restrict-xi-to-GX}. Thus we have:
\begin{prop}\label{prop:val-I-Omega22}
    If $\sigma\otimes \overline{\Theta_{\psi,\chi_1,\chi_2,X,Y}}$ is not $G (Z')$-distinguished for all non-degenerate skew-Hermitian subspace $Z'$ of $X$ such that $Z' \supset Z$ and $\dim Z' = \dim Z +1$, then $I_{\Omega_{2,2}} (\xi)$ vanishes for $\xi = \xi_{c,s}$ and $\xi_s^c$. 
\end{prop}

\section{On Absolute Convergence of Integrals}
\label{sec:absol-conv-integr}

To justify the computation in Sec.~\ref{sec:comp-integrals} where we have freely changed the order of summation and integration, we must verify that the $I_{\Omega_{*,*}} (\xi)$'s \eqref{eq:I-omega} are all absolutely convergent. We recall  an estimate on the growth of theta series. The issue of the absolute convergence of 
$I_{\Omega_{*,*}} (\xi)$ with $(*,*)=(2,2)$  is a case that does not occur in the symplectic/metaplectic case of our previous work \cite{MR3805648,wu:_period-metaplectic}. Thus we have to use new methods to estimate the integrals.
\subsection{Heights}
\label{sec:heights}

For a finite dimensional vector space $V$ over $E$, we fix a basis once and for all. Thus we identify $V(\AA_E)$ with $\AA_E^n$ for some $n$. We define a height function on $\AA_E^n$ as in \cite[Sec.~1.1]{MR0191899}.  Let $(v_1,\ldots , v_n) \in \AA_E^n$. First we have the local heights. For a non-archimedean place $w$ of $E$, set
\begin{equation*}
  ||(v_{1,w},\ldots , v_{n,w})||_w = \max_i \{ |v_{i,w}|_w\}
\end{equation*}
and for a real place $w$, set
\begin{equation*}
  ||(v_{1,w},\ldots , v_{n,w})||_w = (\sum_i |v_{i,w}|_w^2)^{\half}
\end{equation*}
and for a complex place $w$, set
\begin{equation*}
  ||(v_{1,w},\ldots , v_{n,w})||_w = \sum_i |v_{i,w}|_w.
\end{equation*}
Note that the absolute value at a complex place is the square of the Euclidean norm. Then we define
\begin{equation*}
  ||(v_1,\ldots , v_n)|| = \prod_w ||(v_{1,w},\ldots , v_{n,w})||_w.
\end{equation*}

As in \cite[I.2.2]{MR1361168}, we have a height function $||\ ||$ on $G(X_1)(\AA)$, etc.

\subsection{Bound for theta series}
\label{sec:bound-theta-series}

 We need an estimate for the  theta series. For $h\in G (X) (\A)$, let $||h||$ denote the height of $h$ as defined in \cite [I.2.2] {MR1361168}.
\begin{lemma}\label{lemma:theta-bound} Fix $\Phi \in \cS_{X_1,Y}
(\A)$.  There exist $D>0$, $r>0$ and $R>0$ such that
  \begin{equation*} |\theta_{\psi,\chi_1,\chi_2,X_1,Y} (n m_1 (t)h k,1,\Phi)| < D
(|t|^{\frac{\dim Y}{2}}+ |t|^{-r}) || h ||^R
  \end{equation*} for all $n \in N_{1} (\A)$, $t\in \RGL_1 (\A)$, $h\in
  G (X) (\A)$ and $k\in K_{G (X_1)}$.
\end{lemma}
\begin{rmk}
  The proof for the symplectic case (\cite[Lemma~5.3]{MR3805648}) can be easily adapted to the unitary case.
\end{rmk}

\subsection{Absolute Convergence of $I_{\Omega_{0,1}}(\xi)$}
\label{sec:absol-conv-Omega01}

First set $\xi=\xi_{c,s}$. We need to show that \eqref{eq:int-omega01-1} is absolutely convergent. By Lemma.~\ref{lemma:theta-bound}, we just need to show that for $C \in \RR$  and $R>0$
\begin{align*}
  &\int_{K_{G(Z_1)}}  \int_{[G(Z)]} \int_{[\RGL_1]}\int_{[N(Z_1,e_1^-)]} |t|^{s+\rho_{Q_1}-2\rho_{Q(Z_1,e_1^-)}+C}||h||^R \\
    &|f_s (hk)|
  \hat {\tau}_c (H (m_1(t))) dn dt dh dk
\end{align*}
is  convergent.  The integration over $dn$ is over a compact set. This is also the case with the integration over $dk$. The integration over $dt$ is convergent when $\Re s$ is large enough. The integration over $dh$ is convergent since $f_s(\cdot k)$ is in $\sigma$ and is hence rapidly decreasing.

For $\xi=\xi_s^c$, the argument is analogous. We just need to show that
\begin{align*}
&  \int_{K_{G(Z_1)}}  \int_{[G(Z)]} \int_{[\RGL_1]}\int_{[N(Z_1,e_1^-)]} |t|^{-s+\rho_{Q_1}-2\rho_{Q(Z_1,e_1^-)}+C}||h||^R \\
&  |M(w_0,s)f_s (hk)| 
  \hat {\tau}^c (H (m_1(t))) dn dt dh dk
\end{align*}
is convergent. Note that  $C=\half \dim{Y}$ or $-r <0$ as in the exponents of $|t|$ in Lemma.~\ref{lemma:theta-bound}. In any case, as long as $\Re s > -s_0$, the integration over $dt$ is absolutely convergent. We need this finer statement when we consider period integrals of  residues of Eisenstein series.
Thus we find
\begin{lemma}\label{lemma:conv-01}\leavevmode
  \begin{enumerate}
  \item For $\Re s$ large enough,  both $I_{\Omega_{0,1}}(\xi_{c,s})$ and $I_{\Omega_{0,1}}(\xi_s^c)$ are absolutely convergent.
  \item For $\Re s_1 > -s_0$, $I_{\Omega_{0,1}}(\Res_{s=s_1}\xi_s^c)$ is absolutely convergent.
\end{enumerate}
\end{lemma}

\subsection{Absolute Convergence of $I_{\Omega_{1,0}}(\xi)$}
\label{sec:absol-conv-Omega10}

We want to show absolute convergence of \eqref{eq:I_omega10}. First we choose $\gamma_v$ in a more uniform way than in Sec.~\ref{sec:series-over-omega_10}. Fix an isotropic vector $v_0\in V$. The isotropic lines in $V$ can be parametrised by $Q(V,v_0) \lmod G(V)$. Thus we may choose $\gamma_v=\gamma_{v_0}\delta$ for $\delta$ running over a set of representatives of $Q(V,v_0) \lmod G(V)$. After rewriting \eqref{eq:I_omega10}, we see that we   need to show the absolute convergence of
\begin{align*}
  &   \sum_{\delta\in Q(V,v_0)\lmod G(V)} \int_{[G (Z_1)]}  \xi (\gamma_{v_0}\delta g) dg \\
  =  &  \sum_{\delta\in Q(V,v_0)\lmod G(V)} \int_{[G (Z_1)]}  \xi (\gamma_{v_0}g \delta) dg \\
  =  &  \sum_{\delta\in Q(V,v_0)\lmod G(V)} \int_{[G (Ev_0^+ \oplus Z \oplus Ev_0)]}  \xi (g \gamma_{v_0} \delta) dg .
\end{align*}

Let $\xi=\xi_{c,s}$. We have
\begin{align*}
  \sum_{\delta\in Q(V,v_0)\lmod G(V)} \int_{[G (Ev_0^+ \oplus Z \oplus Ev_0)]}
\overline{\theta_{\psi,\chi_1,\chi_2,X_1,Y} (g\gamma_{v_0} \delta, 1, \Phi)} f_s (g\gamma_{v_0} \delta)\hat {\tau}_c (H (\gamma_{v_0} \delta)) dg .
\end{align*}

Write $\gamma_{v_0} \delta = n_{\gamma_{v_0} \delta}m_1(t_{\gamma_{v_0} \delta}) h_{\gamma_{v_0} \delta}k_{\gamma_{v_0} \delta}$ according to the Iwasawa decomposition. Then using the bound for theta series (Lemma~\ref{lemma:theta-bound}), we are reduced to showing the absolute convergence of 
\begin{align*}
  &\sum_{\delta\in Q(V,v_0)\lmod G(V)} \int_{[G (Ev_0^+ \oplus Z \oplus Ev_0)]} |t_{\gamma_{v_0} \delta}|^{s+C}
  ||gh_{\gamma_{v_0} \delta}||^R \\
  &|f_s(gh_{\gamma_{v_0} \delta}k_{\gamma_{v_0} \delta})|\hat {\tau}^c (H (m_1(t_{\gamma_{v_0} \delta}))) dg.
\end{align*}
In fact, we will show the stronger result that
\begin{align*}
  &\sum_{\delta\in Q(V,v_0)\lmod G(V)} \int_{[G (Ev_0^+ \oplus Z \oplus Ev_0)]} |t_{\gamma_{v_0} \delta}|^{s+C}
  ||gh_{\gamma_{v_0} \delta}||^R 
|f_s(gh_{\gamma_{v_0} \delta}k_{\gamma_{v_0} \delta})| dg
\end{align*}
is absolutely convergent. 
We  expand the domain of integration to $[G(X)]$ and then we can absorb $h_{\gamma_{v_0} \delta}$ into $g$. 
 Then due to the rapid decay of $f_s$ on a Siegel domain of $G(X)$,
\begin{equation*}
  \int_{[G (X)]} ||g||^R |f_s(gk)| dg.
\end{equation*}
is convergent and the convergence is uniform in $k\in K_{G(X_1)}$. Thus it remains to show  the absolute convergence of 
\begin{align*}
  \sum_{\delta\in Q(V,v_0)\lmod G(V)} |t_{\gamma_{v_0} \delta}|^{s+C}
\end{align*}
but this is majorised by an Eisenstein series. Hence it is absolutely convergent for $\Re s$ large enough.

Let $\xi=\xi_s^c$. A similar analysis reduces us to showing the absolute convergence of
\begin{align*}
  \sum_{\delta\in Q(V,v_0)\lmod G(V)} |t_{\gamma_{v_0} \delta}|^{-s+C} \hat {\tau}^c (H (m_1(t_{\gamma_{v_0} \delta}))).
\end{align*}
This is a finite sum and thus absolutely convergent. Furthermore, $|t_{\gamma_{v_0} \delta}|$ is comparable to $||e_1^-\gamma_{v_0} \delta||^{-1} = ||v_0\delta||^{-1}\le 1$. Thus if $c$ is large enough, no term survives truncation.

To conclude, we have 
\begin{lemma}\label{lemma:conv-10}\leavevmode
  \begin{enumerate}
  \item For $\Re s$ large enough, $I_{\Omega_{1,0}}(\xi_{c,s})$ is absolutely convergent.
  \item For $c$ large enough and  $s$  not a pole of $M(w_0,s)$, $I_{\Omega_{1,0}}(\xi_s^c)$ is absolutely convergent and is equal to $0$.
  \end{enumerate}
\end{lemma}

\subsection{Absolute Convergence of $I_{\Omega_{1,1}}(\xi)$}
\label{sec:absol-conv-i_omega11}

We need to show absolute convergence of \eqref{eq:I_omege11}. First we choose $\gamma_v$ in a more uniform way than in Sec.~\ref{sec:series-over-omega_11}. Fix an isotropic vector $v_0\in V$. The isotropic lines in $V$ can be parametrised by $Q(V,v_0) \lmod G(V)$. Thus we may choose $\gamma_{v+e_1^-}=\gamma_{v_0+e_1^-}\delta$ for $\delta$ running over a set of representatives of $Q(V,v_0) \lmod G(V)$. Thus we arrive at
\begin{align*}
  &  \sum_{\delta\in Q(V,v_0)\lmod G(V)} \int_{K_{G(Z_1)}} \int_{\RGL_1(\AA)}  \int_{[J(Z_1,e_1^-)]}    \xi (\gamma_{v_0+e_1^-}\delta g m_1(t)k) \\
  &|t|^{-2\rho_{Q(Z_1,e_1^-)}}dg dt dk \\
  =&    \sum_{\delta\in Q(V,v_0)\lmod G(V)} \int_{K_{G(Z_1)}} \int_{\RGL_1(\AA)}  \int_{[J(Ev_0^+\oplus Z \oplus Ev_0, v_0)]}    \xi ( g\gamma_{v_0+e_1^-} \delta m_1(t)k) \\
  &|t|^{-2\rho_{Q(Z_1,e_1^-)}} dg dt dk .
\end{align*}
We use the Iwasawa decomposition
\begin{equation*}
  \gamma_{v_0+e_1^-} \delta m_1(t) = m_1(t_*) h_* n_* k_*
\end{equation*}
where $*$ denotes $\gamma_{v_0+e_1^-} \delta m_1(t)$ as this is quite lengthy.

Let $\xi=\xi_{c,s}$. Then by the bound on theta series (Lemma~\ref{lemma:theta-bound}), it reduces to showing the absolute convergence of
\begin{align*}
  &\sum_{\delta\in Q(V,v_0)\lmod G(V)} \int_{K_{G(Z_1)}} \int_{\RGL_1(\AA)}  \int_{[J(Ev_0^+\oplus Z \oplus Ev_0, v_0)]}
  |t_*|^{s+C} ||gh_*||^R\\
  &|f_s(gh_*k_*k)| |t|^{-2\rho_{Q(Z_1,e_1^-)}} \hat{\tau}_c(H(m_1(t_*))) dg dt dk.
\end{align*}
After enlarging the domain of integration of $dg$ from $[J(Ev_0^+\oplus Z \oplus Ev_0, v_0)]$ to $[G(X)]$, we may absorb $h_*$ into $g$ and we have the absolute convergence (uniform in $k\in K_{G(X_1)}$) of
\begin{equation*}
  \int_{[G(X)]} ||g||^R |f_s(g k)| dg
\end{equation*}
due to the rapid decay of $f_s(\cdot k)$ on $[G(X)]$. Thus it suffices to show the absolute convergence of
\begin{equation*}
  \sum_{\delta\in Q(V,v_0)\lmod G(V)} \int_{\RGL_1(\AA)} |t_*|^{s+C} |t|^{-C_2} \hat{\tau}_c(H(m_1(t_*))) dt
\end{equation*}
for fixed constants $C\in\RR$ and $C_2>1$.
Note that $t_*$ depends on $\delta$ and $t$. The quantity $|t_*|$ is comparable to $||e_1^- \gamma_{v_0+e_1^-} \delta m_1(t)||^{-1} = ||v_0\delta + e_1^- \bar{t}^{-1}||^{-1}$. Thus we now try to control
\begin{equation}\label{eq:omega11-sum-delta-int-RGL1}
  \sum_{\delta\in Q(V,v_0)\lmod G(V)} \int_{\RGL_1(\AA)} ||v_0\delta + e_1^- \bar{t}^{-1}||^{-(s+C)} |t|^{-C_2} dt.
\end{equation}
The  integral in \eqref{eq:omega11-sum-delta-int-RGL1} can be decomposed into a product over all places $w$ of $E$ of the local integrals
\begin{equation*}
  \int_{E_w^\times} ||v_0\delta + e_1^- t_w^{-1}||_w^{-(s+C)}  |t_w|_w^{-C_2}  d^\times t_w.
\end{equation*}
Here we used the fact that if $w$ and $\bar{w}$ are two places of $E$ lying over a split place $v$ of $F$, then $|t_w|_{\bar{w}} = |t_w|_w$ where we view $t_w$ as an element of $F_v\isom E_w \isom E_{\bar{w}}$ and if $w$ is a place of $E$ lying over a non-split place $v$ of $F$, then $|\overline{t_w}|_w = |t_w|_w$ where bar denotes the non-trivial Galois action of $E_w/F_v$. Let $A_w=||v_0\delta||_w$. Then if $w$ is non-archimedean, we get
\begin{align*}
  &  \int_{E_w^\times} \max\{A_w, |t_w|_w^{-1} \}^{-(s+C)}  |t_w|_w^{-C_2}  d^\times t_w\\
  =& A_w^{-(s+C)}\int_{E_w^\times} \max\{1, A_w^{-1}|t_w|_w^{-1} \}^{-(s+C)}  |t_w|_w^{-C_2}  d^\times t_w\\
    =& A_w^{-(s+C-C_2)}\int_{E_w^\times} \max\{1, |t_w|_w^{-1} \}^{-(s+C)}  |t_w|_w^{-C_2}  d^\times t_w.
\end{align*}
We have changed variable at the last step.
We compute the integral
\begin{align*}
  &  \int_{E_w^\times} \max\{1, |t_w|_w^{-1} \}^{-(s+C)}  |t_w|_w^{-C_2}  d^\times t_w 
  = \sum_{n=-\infty}^{-1} q_w^{nC_2} + \sum_{n=0}^{+\infty} q_w^{-n(s+C-C_2)}\\
  =& \frac{q_w^{-C_2}}{1- q_w^{-C_2}} + \frac{1}{1-q_w^{-(s+C-C_2)}}
  =  \frac{1-q_w^{-(s+C)}}{(1-q_w^{-C_2})(1-q_w^{-(s+C-C_2)})}
\end{align*}
where $q_w$ is the cardinality of the residue field.
If $w$ is real, the local integral is
\begin{align*}
  &  \int_{E_w^\times} (A_w^2 + |t_w|_w^{-2})^{-\half(s+C)} |t_w|_w^{-C_2} d^\times t_w \\
  =& A_w^{-(s+C-C_2)} \int_{E_w^\times} (1 + |t_w|_w^{-2})^{-\half(s+C)} |t_w|_w^{-C_2} d^\times t_w .
\end{align*}
If $w$ is complex, the local integral is
\begin{align*}
  &  \int_{E_w^\times} (A_w + |t_w|_w^{-1})^{-(s+C)} |t_w|_w^{-C_2} d^\times t_w \\
  =& A_w^{-(s+C-C_2)} \int_{E_w^\times} (1 + |t_w|_w^{-1})^{-(s+C)} |t_w|_w^{-C_2} d^\times t_w .
\end{align*}
In both cases, when $\Re s$ is large enough, the integrals are absolutely convergent.  Combining these we find that \eqref{eq:omega11-sum-delta-int-RGL1} is equal to the product of  absolutely convergent integrals at the archimedean places and 
\begin{align*}
  \sum_{\delta\in Q(V,v_0)\lmod G(V)} ||v_0\delta||^{-(s+C-C_2)} \frac{L_E^{S_\infty}(C_2)L_E^{S_\infty}(s+C-C_2)}{L_E^{S_\infty}(s+C)}.
\end{align*}
Here the $L$-functions are partial Dedekind $L$-functions away from the  archimedean places.
As long as $\Re s$ is large enough, the part consisting of the partial $L$-functions is absolutely convergent. We note that $C_2>1$. The series over $\delta$ is bounded by an Eisenstein series and thus, when $\Re s$ is large enough, is absolutely convergent. 

Let $\xi=\xi_s^c$. A similar analysis shows that it suffices to prove the absolute convergence of
\begin{align*}
  \sum_{\delta\in Q(V,v_0)\lmod G(V)} \int_{\RGL_1(\AA)} |t_*|^{-s+C} |t|^{-C_2} \hat{\tau}^c(H(m_1(t_*))) dt.
\end{align*}
Noting that $|t_*|^{-1}$ is comparable to 
\begin{equation*}
  ||v_0\delta + e_1^- \bar{t}^{-1}|| \ge ||v_0\delta|| \ge 1
\end{equation*}
and only those $t_*$ with $|t_*|< c^{-1}$ can survive the truncation,
we see that if $c$ is chosen large enough, then nothing remains.

In conclusion, we have
\begin{lemma}\label{lemma:conv-11}\leavevmode
  \begin{enumerate}
  \item For $\Re s$ large enough, $I_{\Omega_{1,1}}(\xi_{c,s})$ is absolutely convergent.
  \item For $c$ large enough and $s$ not a pole of $M(w_0,s)$, $I_{\Omega_{1,1}}(\xi_s^c)$ is absolutely convergent and is equal to $0$.
  \end{enumerate}
\end{lemma}

\subsection{Absolute Convergence of $I_{\Omega_{2,2}}(\xi)$}
\label{sec:absol-conv-i_omega22}

We need to show the absolute convergence of \eqref{eq:I-Omega-22}.
The Bruhat decomposition of $\gamma_{v+z_\alpha}$ is given as follows
\begin{align}\label{eq:omega22-gamma-bruhat}
  \begin{pmatrix}
    1 & \frac{1}{\alpha} & -\frac{1}{2\alpha} \\
    0 & 1 & -1 \\
    0 & 0 & 1
  \end{pmatrix}
            \begin{pmatrix}
              &&1 \\ &1& \\ -1&&
            \end{pmatrix}
                                 \begin{pmatrix}
                                   \frac{\alpha}{2} & 0 & 0 \\
                                   0 & 1 & 0 \\
                                   0 & 0 & -\frac{2}{\alpha}
                                 \end{pmatrix}
                                           \begin{pmatrix}
                                             1 & -\frac{2}{\alpha} & -\frac{2}{\alpha} \\
                                             0 & 1 & 2 \\
                                             0 & 0 & 1
                                           \end{pmatrix}
\end{align}
with respect to the basis $e_1^+, v, e_1^-$. The four matrices will be denoted by $\gamma_{v+z_\alpha}^{(1)}$, $w_0$, $\gamma_{v+z_\alpha}^{(2)}$,  $\gamma_{v+z_\alpha}^{(3)}$ respectively and if there is no confusion, the subscript $v+z_\alpha$ will be omitted.

Assume that $\xi=\xi_{c,s}$. First we bound one term of the series \eqref{eq:I-Omega-22}, i.e.,
\begin{equation}\label{eq:omega22-one-term}
  \int_{ G (Z\oplus Ez_\alpha') (F)\lmod G (Z_1) (\A)} \xi (\gamma_{v+z_\alpha} g) dg.
\end{equation}
Formally \eqref{eq:omega22-one-term} is equal to
\begin{align*}
  &\int_{G (Z\oplus Ez_\alpha') (\AA)\lmod G (Z_1) (\A)} \int_{[G (Z\oplus Ez_\alpha')]}\overline{\theta_{\psi,\chi_1,\chi_2,X_1,Y} (\gamma_{v+z_\alpha} hg, 1, \Phi)} \\
  &\quad f_s (\gamma_{v+z_\alpha} hg)\hat {\tau}_c (H (\gamma_{v+z_\alpha} hg))  dh dg \\
  =&\int_{G (Z\oplus Ez_\alpha') (\AA)\lmod G (Z_1) (\A)} \int_{[G (Z\oplus Ev)]}  \overline{\theta_{\psi,\chi_1,\chi_2,X_1,Y} (h\gamma_{v+z_\alpha} g, 1, \Phi)} \\
  &\quad f_s (h\gamma_{v+z_\alpha} g)\hat {\tau}_c (H (\gamma_{v+z_\alpha} g))  dh dg  .
\end{align*}
  Applying the Bruhat decomposition \eqref{eq:omega22-gamma-bruhat} and noting that $h\gamma^{(1)}h^{-1}$ lies in $N_1 (\AA)$, we get
\begin{align*}
  \int_{G (Z\oplus Ez_\alpha') (\AA)\lmod G (Z_1) (\A)}\int_{[G (Z\oplus Ev)]}  \overline{\theta_{\psi,\chi_1,\chi_2,X_1,Y} (h\gamma^{(1)}h^{-1} h w_0\gamma^{(3)} g, 1, \Phi)}\\
    f_s (h w_0\gamma^{(3)} g ) \hat {\tau}_c (H ( w_0\gamma^{(3)} g))  dh dg  .
\end{align*}
As $Q (Z_1,Ee_1^-) (\AA) \cap G (Z\oplus Ez_\alpha') (\AA) = G (Z) (\AA)$,  it suffices to bound
\begin{align*}
  \int_{\RGL_1 (\AA)} \int_{N (Z_1,Ee_1^-) (\AA)} \int_{[G (Z\oplus Ev)]}  \overline{\theta_{\psi,\chi_1,\chi_2,X_1,Y} (h\gamma^{(1)}h^{-1} h w_0n\gamma^{(3)}  m_1 (t)  , 1, \Phi)}\\
  f_s (h w_0n\gamma^{(3)} m_1 (t)  ) \hat {\tau}_c (H ( w_0n\gamma^{(3)} m_1 (t) )) |t|^{-2\rho_{Q (Z_1)}}  dh dn dt  .
\end{align*}

Observe that the integration over $n$ is similar to the definition of the intertwining operator from $\cA _1 (s,\chi_1,\sigma)$ to $\cA _1 (-s,\chi_1,\sigma)$. It suffices to analyse
\begin{align*}
&  \sum_{\mu\in \RGL_1(F)}\int_{[\RGL_1]} \int_{[N (Z_1,Ee_1^-) ]} \int_{[G (Z\oplus Ev)]} \sum_{\delta\in N (Z_1,Ee_1^-) }\\
 & \quad \overline{\theta_{\psi,\chi_1,\chi_2,X_1,Y} (h\gamma^{(1)}h^{-1} h w_0\delta n\gamma^{(3)}  m_1 (\mu t)  , 1, \Phi)}\\
  &\quad f_s (h w_0\delta n\gamma^{(3)} m_1 (\mu t)  ) \hat {\tau}_c (H ( w_0\delta n\gamma^{(3)} m_1 (\mu t) )) |t|^{-2\rho_{Q (Z_1)}}  dh dn dt  .
\end{align*}
Now the integration over $n$ is a compact integral.

We bring back the whole series in \eqref{eq:I-Omega-22} and we need  to determine the absolute convergence of
\begin{align}\label{eq:four-series}
  &\sum_{\alpha\in\daleth F^\times/\Nm_{E/F}E^\times} \sum_{v\in V, \form{v}{v}=\alpha} \sum_{\delta\in N (Z_1,Ee_1^-) } \sum_{\mu\in \RGL_1(F)} \int_{[\RGL_1]} \nonumber\\
  &\quad  \int_{[N (Z_1,Ee_1^-) ]} \int_{[G (Z\oplus Ev)]} \overline{\theta_{\psi,\chi_1,\chi_2,X_1,Y} (h\gamma^{(1)}h^{-1} h w_0\delta \gamma^{(3)}m_1(\mu)  m_1 (t) n , 1, \Phi)} \nonumber\\
  &\quad f_s (h w_0\delta \gamma^{(3)}m_1(\mu) m_1 (t) n ) \hat {\tau}_c (H ( w_0\delta \gamma^{(3)}m_1(\mu) m_1 (t) n))  dh dn dt 
\end{align}
or, a fortiori,
\begin{align*}
  &\sum_{\gamma \in N (X_1,Ee_1^-) \setminus \{I\} } \int_{[\RGL_1]} \int_{[N (Z_1,Ee_1^-) ]} \int_{[G (Z\oplus Ev)]} \\
  &\overline{\theta_{\psi,\chi_1,\chi_2,X_1,Y} (h\gamma^{(1)}h^{-1} h w_0 \gamma  m_1 (t) n , 1, \Phi)} f_s (h w_0\gamma  m_1 (t) n )\\
 &   \hat {\tau}_c (H ( w_0 \gamma m_1 (t) n))  dh dn dt  .
\end{align*}
Observe that the four series in \eqref{eq:four-series} together parametrise a certain set of isotropic lines in $X$ and we have expanded this set.   

Define the function $\eta: G(X_1)(\AA) \rightarrow \RGL_1(\AA)/\RGL_1(\AA)^1$ such that if we write $g\in G(X_1)(\AA)$ according to the Iwasawa decomposition 
\begin{equation*}
  g= n m_1(t)h k
\end{equation*}
for $n\in N_1(\AA)$, $t\in \RGL_1(\AA)$, $h\in G(X)(\AA)$ and $k\in K_{G(X_1)}$, then $g$ is sent to the coset of $t$.

By Lemma~\ref{lemma:theta-bound}, the theta term has the bound which is a constant times
\begin{equation*}
  (|\eta(w_0\gamma  m_1 (t)) |^{\dim Y/2} +  |\eta(w_0\gamma  m_1 (t)) |^{-C_1} )||hh'||^{C_2}
\end{equation*}
for some constants $C_1,C_2>0$
where we write
\begin{equation}\label{eq:iwasawa-w0-gamma-m1t}
  w_0 \gamma m_1 (t) = n' m_1 (t') h' k'
\end{equation}
according to the Iwasawa decomposition. We note that $h'$ can be taken to be the identity in our particular case. 

Thus we arrive at determining the absolute convergence for
\begin{align*}
  &\sum_{\gamma \in N (X_1,Ee_1^-) \setminus \{I\} } \int_{[\RGL_1]} \int_{[N (Z_1,Ee_1^-) ]} \int_{[G (Z\oplus Ev)]}
  |\eta(w_0\gamma  m_1 (t)) |^{s+C_3} ||h||^{C_2}\\
  &\qquad f_s (hn)  dh dn dt  .
\end{align*}
The inner integral over $h$ is absolutely convergent due to the cuspidality of $\sigma$ and hence the rapid decay of $f_s (hn)$. Thus we just need to show the absolute convergence of 
\begin{align}\label{eq:omega22-int-over-t}
  \sum_{\gamma \in N (X_1,Ee_1^-) \setminus \{I\} }  \int_{[\RGL_1]} |\eta(w_0\gamma  m_1 (t)) |^{s+C_3} dt.
\end{align}
We exchange the order of summation and integration. Consider the series
\begin{equation*}
  g \mapsto \sum_{\gamma \in N (X_1,Ee_1^-) \setminus \{I\} } |\eta(w_0\gamma  g) |^{s+C_3}.
\end{equation*}
This is almost an Eisenstein series associated to the parabolic $Q_1$ and the section $F_s: g\mapsto |\eta(g)|^{s+C_3}$. Set $C_4=C_3-\rho_{Q_1}$ and $s'=s+C_4$. Then the series is 
\begin{equation}\label{eq:almost-Eis}
  E(g,F_{s'}) - F_{s'}(g) - F_{s'}(w_0g).
\end{equation}
We show that  the integral of this over $[\RGL_1]$ is absolutely convergent when $\Re s$ is large enough.
By \cite[Corollary~I.2.12]{MR1361168}, $E(g,F_{s'})-F_{s'}(g)-M(w_0,s')F_{s'}(g)$ is rapidly decreasing. In particular, if we restrict to $g=m_1(t)$, then it is rapidly decreasing when $|t|\rightarrow \infty$. Hence \eqref{eq:almost-Eis}, which is equal to 
\begin{align*}
  &   E(m_1(t),F_{s'}) - F_{s'}(m_1(t)) -M(w_0,s')F_{s'}(m_1(t)) \\
  & \quad -F_{s'}(w_0m_1(t)) + M(w_0,s')F_{s'}(m_1(t)),
\end{align*}
 decays like $|t|^{-s'+\rho_{Q_1}}$ when $|t|\rightarrow \infty$. On the other hand, if we set $g=w_0m_1(t)w_0^{-1} = m_1(t^{-1})$, then 
\begin{equation*}
  E(w_0m_1(t)w_0^{-1} ,F_{s'})-F_{s'}(w_0m_1(t)w_0^{-1} )-M(w_0,s')F_{s'}(w_0m_1(t)w_0^{-1})
\end{equation*}
is rapidly decreasing when $|t|\rightarrow 0$. Using the fact that the inducing section is unramified, we find that the above quantity is equal to
\begin{equation*}
  E(m_1(t) ,F_{s'})-F_{s'}(w_0m_1(t) )-M(w_0,s')F_{s'}(w_0m_1(t)).
\end{equation*}
Hence \eqref{eq:almost-Eis}, which is equal to 
\begin{align*}
  &E(m_1(t),F_{s'}) - F_{s'}(m_1(t)) - F_{s'}(w_0m_1(t))\\
  &-M(w_0,s')F_{s'}(w_0m_1(t)) +M(w_0,s')F_{s'}(w_0m_1(t)),
\end{align*}
decays like $|t|^{s'-\rho_{Q_1}}$ as $|t|\rightarrow 0$. We have shown that the integrand $E(m_1(t),F_{s'}) - F_{s'}(m_1(t)) - F_{s'}(w_0m_1(t))$ in \eqref{eq:omega22-int-over-t} decays like $|t|^{-s'+\rho_{Q_1}}$ as $|t|\rightarrow \infty$ and decays like $|t|^{s'-\rho_{Q_1}}$ as $|t|\rightarrow 0$. Thus the integral is absolutely convergent.

Next we set $\xi=\xi_s^c$. Using an analogous analysis to the case $\xi=\xi_{c,s}$, we reach the point when we need to determine the absolute convergence of 
\begin{equation}\label{eq:omega22-int-t-upper-trunc}
  \sum_{\gamma \in N(X_1,e_1^{-})\setminus \{I\} }  \int_{[\RGL_1]} |\eta(w_0\gamma  m_1 (t)) |^{-s+C_4} \hat{\tau}^c(H(w_0\gamma  m_1 (t))) dt
\end{equation}
for some constant $C_4$. Here we have kept the truncation.
Recall that the set $N(X_1,e_1^{-})\setminus \{I\}$ we sum over is enlarged from the set that is parametrised by the four series in \eqref{eq:four-series}. We may use a  smaller set as long as it contains the set parametrised by the four series in \eqref{eq:four-series}. We explain which set we use.
Recall that
 $N(X_1,e_1^-)$ is isomorphic to $\Hom_E (\ell_1^+, X) \rtimes \Herm_1$.  We parametrise
 $n\in N(X_1,e_1^-) $   as $n(\mu,\beta)$ for $\mu\in \Hom_E (\ell_1^+, X)$ and $\beta\in \Herm_1$.
Observing $\gamma^{(3)}$ always has a non-zero `$\mu$-part', we may consider the series over the set which consists of elements $n(\mu,\beta)$ with $\mu\neq 0$ in $N(X_1,e_1^{-}) \setminus \{I\}$. 
As $\exp(H(w_0\gamma  m_1 (t)))$ is comparable to $||e_1^- w_0\gamma  m_1 (t)||^{-1}$, only those terms such that $||e_1^- w_0\gamma  m_1 (t)|| \le c^{-1}$ can possibly contribute. Writing $\gamma=n(\mu,\beta) \in N(X_1,e_1^{-})\setminus \{I\}$ with $\mu\neq 0$ and setting
$\beta'= \beta - \half \mu\mu^*$, we have
\begin{equation*}
  ||e_1^- w_0\gamma  m_1 (t)|| = ||(t,\mu,\bar{t}^{-1}\beta')|| \ge ||\mu|| \ge 1.
\end{equation*}
This means if $c$ is large enough, nothing survives and \eqref{eq:omega22-int-t-upper-trunc} is zero. Hence we have proven

\begin{lemma}\label{lemma:conv-22}\leavevmode
  \begin{enumerate}
  \item For $s$ with $\Re s$ large enough,  $I(\xi_{c,s})$ is
    absolutely convergent.
  \item For $c$ large enough and  $s$  not a pole of $M(w_0,s)$,   $I(\xi_s^c)$ is absolutely convergent and is equal to zero.
\end{enumerate}
\end{lemma}

\bibliographystyle{alpha}

\end{document}